\documentclass[11pt, oneside]{article}
\usepackage{amsfonts}
\usepackage{mathrsfs}
\usepackage{color}
\usepackage{latexsym}
\usepackage{amssymb}
\usepackage{amsmath}
\usepackage{enumerate}
\usepackage{amsthm}
\usepackage{indentfirst}
\usepackage{mathtools}
\usepackage{tikz}
\usepackage{psfrag}
\usepackage{graphicx}
\usepackage{bm}
\usepackage[rflt]{floatflt}
\usepackage{float}
\usepackage{authblk}
\usepackage{ulem}
\usepackage[square, comma, sort&compress, numbers]{natbib}
\usepackage{verbatim}
\usepackage{amsthm}
\usepackage{esint}

\numberwithin{equation}{section}
\allowdisplaybreaks

\newenvironment{proof2.1}{\medskip\noindent{\bf Proof of the Theorem 2.1:}\enspace}{\hfill \qed \newline \medskip}
\newenvironment{proof2.2}{\medskip\noindent{\bf Proof of the Theorem 2.2:}\enspace}{\hfill \qed \newline \medskip}

\newtheorem{theorem}{\color{black}\indent Theorem}[section]
\newtheorem{lemma}{\color{black}\indent Lemma}[section]
\newtheorem{proposition}{\color{black}\indent Proposition}[section]
\newtheorem{definition}{\color{black}\indent Definition}[section]
\newtheorem{remark}{\color{black}\indent Remark}[section]

\usepackage{amssymb,amsmath}
\begin{document}
\title{Existence and multiplicity of solutions for a critical Kirchhoff type elliptic equation with a logarithmic perturbation}
\author{Qian Zhang\qquad Yuzhu Han$^{\dag}$}

\affil{School of Mathematics, Jilin University,
 Changchun 130012, P.R. China}
\renewcommand*{\Affilfont}{\small\it}
\date{} \maketitle
\vspace{-20pt}

\footnotetext{\hspace{-1.9mm}$^\dag$Corresponding author.\\
Email addresses: yzhan@jlu.edu.cn(Y. Han).

\thanks{
$^*$Supported by the National Key Research and Development Program of China
(grant no.2020YFA0714101).}}

{\bf Abstract}
In this paper, we are interested in the following critical Kirchhoff type elliptic equation with
a logarithmic perturbation
\begin{equation}\label{eq0}
\begin{cases}
-\left(1+b\int_{\Omega}|\nabla{u}|^2\mathrm{d}x\right) \Delta{u}=\lambda u+\mu u\log{u^2}+|u|^{2^{*}-2}u,
&x\in\Omega,\\
u=0,&x\in\partial\Omega,
\end{cases}
\end{equation}
where $\Omega$ is a bounded domain in $\mathbb{R}^{N}(N\geq3)$ with smooth boundary $\partial \Omega$,
$b$, $\lambda$ and $\mu$ are parameters and $2^{*}=\frac{2N}{N-2}$ is the critical Sobolev exponent.
The presence of a nonlocal term, together with a critical nonlinearity and a logarithmic term, prevents
to apply in a straightforward way the classical critical point theory. Moreover, the geometry structure
of the energy functional changes as the space dimension $N$ varies, which has
a crucial influence on the existence of solutions to the problem. On the basis of some careful analysis
on the structure of the energy functional, existence and (or) multiplicity results are obtained by
using variational methods. More precisely, if $N=3$, problem \eqref{eq0} admits a local minimum solution,
a ground state solution and a sequence of solutions with their $H_0^1(\Omega)$-norms converging to $0$.
If $N=4$, the existence of infinitely many solutions is also obtained. When $N\geq5$, problem \eqref{eq0}
admits a local minimum solution with negative energy. Sufficient conditions are also derived for the
local minimum solution to be a ground state solution.



\section{Introduction and the main results}

In this paper, we consider the following critical Kirchhoff type elliptic problem with a logarithmic
perturbation
\begin{equation}\label{eq1}
\begin{cases}
-\left(1+b\int_{\Omega}|\nabla{u}|^2\mathrm{d}x\right) \Delta{u}=\lambda u+\mu u\log{u^2}+|u|^{2^{*}-2}u,
&x\in\Omega,\\
u=0,&x\in\partial\Omega,
\end{cases}
\end{equation}
where $\Omega$ is a bounded domain in $\mathbb{R}^N(N\geq3)$ with a smooth boundary $\partial\Omega$,
$b>0$, $\lambda \in\mathbb{R}$, $\mu<0$, and $2^{*}=\frac{2N}{N-2}$ is the critical Sobolev exponent
for the embedding $H_0^1(\Omega)\hookrightarrow L^{p}(\Omega)$.

Problem \eqref{eq1} is closely related to the stationary state of the following Kirchhoff type equation
\begin{align}\label{kirchhoff1}
\rho\frac{\partial^2 u}{\partial t^2}-\left(\frac{P_0}{h}+\frac{E}{2L}\int_0^L\left|\frac{\partial u}
{\partial x}\right|^2\mathrm{d}x\right)
\frac{\partial^2 u}{\partial x^2}=f(x,u),
\end{align}
proposed by Kirchhoff \cite{Kirchhoff1883} as a generalization of the classical d'Alembert's
wave equation by considering the effects of the changes in the length of the strings during
the vibrations. Here $\rho$ is the mass density, $P_0$ is the initial tension, $h$ is the area
of cross-section, $E$ is the Young modulus of the material, $L$ is the length of the string
and $f$ is the external force. Kirchhoff equations, also referred to as nonlocal equations,
have their origins in the theory of nonlinear vibration and have been widely applied in physics,
biological, engineering and other fields. For more details on the background of Kirchhoff type
equations, we refer the interested readers to \cite{Arosio,Cousin,Herrmann2024} and the references
therein.

During the past few decades, the analysis of the stationary problems of equation \eqref{kirchhoff1}
with different nonlinearities has been receiving considerable attention. When the nonlinearity is
of the form $f(x,s)=|s|^{2^{*}-2}s$, the interaction between the nonlocal term and the critical term
prevents to apply in a straightforward way the classical critical point theory and makes the study
of such problems not only challenging but also interesting. In view of this, many authors have devoted
themselves to the study of Kirchhoff type elliptic equations and remarkable progress has been made on
the existence, nonexistence and multiplicity of weak solutions by using variational and topological
methods. In particular, problem \eqref{eq1} with $\mu=0$, i.e., the following problem
\begin{equation}\label{eq2}
\begin{cases}
-\left(1+b\int_{\Omega}|\nabla{u}|^2\mathrm{d}x\right) \Delta{u}=\lambda u+|u|^{2^{*}-2}u,&x\in\Omega,\\
u=0,&x\in\partial\Omega
\end{cases}
\end{equation}
has been investigated by several authors. For example, when $N=3$, Naimen \cite{Naimen-3} considered the
existence of solutions to problem \eqref{eq2} by regarding $\lambda$ as a given constant and $b$ as a
parameter. He showed that if $\lambda\leq\lambda_1/4$, problem \eqref{eq2} has no solution for all $b>0$.
If $\lambda_1/4<\lambda\leq \lambda_1$, problem \eqref{eq2} has a mountain pass type solution for $b>0$
suitably small while it has no solution for $b>0$ sufficiently large. If $\lambda>\lambda_1$, problem
\eqref{eq2} has a local minimum solution when $b>0$ is sufficiently large. It may also admit a mountain
pass type solution when $\lambda$ lies in a right neighborhood of $\lambda_1$ and $b$ fulfills some
additional conditions. Here $\lambda_1$ is the first eigenvalue of $-\Delta$ in $H_0^1(\Omega)$.
When $N=4$, Naimen \cite{Naimen-4} proved that if $0<\lambda <\lambda_1$, problem \eqref{eq2} possesses
a mountain pass type solution if and only if $bS^2<1$, where $S>0$ is the best Sobolev embedding constant
given in \eqref{S}. In the high dimension case $N\geq5$, Naimen et al. \cite{Naimen-5} showed that
problem \eqref{eq2} admits at least two solutions (one is a local minimum solution and the other is a
mountain pass type solution) when $0<\lambda<\lambda_1$ and $b>0$ is suitably small. For more results
on the existence, nonexistence and multiplicity of solutions to Kirchhoff type elliptic equations
with other nonlinearities, interested readers may refer to \cite{Faraci,F-2013,Li2012,Perera2006,Wang2012,
zhang2024,Figueiredo2025,Ambrosio2024} and the references therein.

On the other hand, partial differential equations (PDEs for short) with logarithmic type perturbation
have also been investigated extensively in recent years. One reason is that PDEs involving logarithmic
terms can be used to describe many interesting phenomena in physics and other applied sciences such
as the viscoelastic mechanics and quantum mechanics theory, see \cite{Squassina, Shuaiwei2023}.
Another reason is that different from the power type nonlinearity, the sign of the logarithmic term $u\log u^2$
is indefinite for $u>0$ and it does not satisfy the standard Aimbrosetti-Rabinowitz
condition. Moreover, the logarithmic term in $\eqref{eq1}$ has the following properties
\begin{align*}
\lim_{u\rightarrow0^{+}}\frac{u\log u^2}{u}=-\infty, \ \ \ \ \
\ \ \ \lim_{u\rightarrow +\infty}\frac{u\log u^2}{|u|^{2^{*}-2}u}=0,
\end{align*}
which means that $u=o(u\log u^2)$ for $u$ very close to $0$, and that the logarithmic term $u\log u^2$
is a lower-order term at infinity compared with the critical term $|u|^{2^{*}-2}u$. All these properties
make the structure of the corresponding functional more complicated and bring essential difficulties in
looking for weak solutions to such problems in the variational framework. When $b=0$,
problem \eqref{eq1} is reduced to the following Br\'{e}zis-Nirenberg problem with a logarithmic
perturbation
\begin{equation}\label{eq3}
\begin{cases}
-\Delta{u}=\lambda u+\mu u\log{u^2}+|u|^{2^{*}-2}u,&x\in\Omega,\\
u=0,&x\in\partial\Omega.
\end{cases}
\end{equation}
With the help of Mountain Pass Lemma and some delicate estimates on the logarithmic term, Deng et al.
\cite{Deng2023} proved the existence of a mountain pass type solution (which is also a ground state
solution) to problem \eqref{eq3} when $\lambda\in \mathbb{R}$, $\mu>0$ and $N\geq4$. When $\mu<0$,
the corresponding energy functional possesses quite different variational properties and whether or
not the energy functional is still locally compact is still unknown. By applying the
mountain pass theorem without local $(PS)$ condition they obtained a weak solution under suitable
assumptions on $\lambda$ and $\mu$ when $N=3,4$. Moreover, when $N\geq3$, a nonexistence result
was also obtained for $\mu<0$. Later, Hajaiej et al. \cite{Hajaiej2024,Ha2024} reconsidered problem
\eqref{eq3} with $\mu<0$. They not only weakened part of the existence conditions, but also specified
the types of the weak solutions when $N\geq4$. Recently, by using the subcritical approximation method,
the existence of sign-changing solutions to problem \eqref{eq3} with $N\geq6, \lambda\in \mathbb{R}$
and $\mu>0$ was obtained by Liu et al. in \cite{Liu2024}. For more results on the existence and
multiplicity of solutions to PDEs with logarithmic type nonlinearities, we refer the interested readers to
\cite{Zhang2023, Shen2025,yang2022,LHW2023,Troy2016} and the references therein

To the best of our knowledge, there are few works dealing with the existence and multiplicity of weak
solutions to problem \eqref{eq1} with $\mu\neq 0$. When $N=4$ and $\mu>0$, by applying Jeanjean's monotonicity
result \cite{Jeanjean} and an approximation method, the second author of this paper \cite{Han2024} proved
that either the $H_0^1(\Omega)$-norm of the sequence of approximate solutions diverges to infinity or
the problem \eqref{eq1} admits a nontrivial weak solution. Furthermore, the former case can be excluded
when $\Omega$ is star-shaped. A key point in the process is the local compactness
of the energy functional, the proof of which depends heavily on the assumption $\mu>0$ and the fact that the
exponent corresponding to the critical term is exactly the same as that corresponding to the nonlocal term
in the energy functional. However, the situation becomes more thorny when $\mu<0$, and the approach used in
\cite{Han2024} to obtain the local compactness of the $(PS)$ sequence cannot be applied here. Moreover,
the topological properties of the energy functional also change and $0$ is no longer a local minimum of the
energy functional.

Motivated by the works mentioned above, we consider problem \eqref{eq1} with $\mu<0$ and investigate how
the interaction between the nonlocal term, the critical nonlinearity and the logarithmic term affect the
existence and (or) multiplicity of weak solutions to this problem. Since the geometry structure of the energy
functional changes dramatically as the space dimension $N$ varies, our discussion is divided into three
cases, i.e., $N=3$, $N=4$ and $N\geq 5$.

When $N=3$, the energy functional possesses the local minimum structure and mountain pass structure under
suitable assumptions. In view of this, we apply Ekeland's variational principle, Br\'{e}zis-Lieb's lemma
and some estimates to prove that problem \eqref{eq1} admits a local minimum solution and a ground state
solution. Furthermore, it is shown under some more assumptions on $b$, $\lambda$ and $\mu$ that the local
minimum solution is also a ground state solution. Moreover, by combining Kajikiya's symmetric mountain
pass lemma with an appropriate truncation on the critical term, we prove that problem \eqref{eq1} possesses
a sequence of solutions whose energies and $H_0^1(\Omega)$-norms converge to $0$. Set
\begin{align*}
\mathcal{A}_0:&=\left\{( b, \lambda, \mu): b>0, \lambda\in \mathbb{R}, \mu<0, \frac{bS^3}{12}+\frac{(4+b^2S^3)(bS^3+\sqrt{b^2S^6+4S^3})}{24}+\frac{\mu}{2}e^{-\frac{\lambda}{\mu}}|\Omega|>0 \right\},\\
\mathcal{A}_1:&=\left\{( b, \lambda, \mu): b>0, \lambda\in \mathbb{R}, \mu<0,
\frac{1}{3}C_{b,S}+\frac{b}{12}C_{b,S}^2+\frac{\sqrt{e}}{3}\mu e^{-\frac{\lambda}{\mu}}|\Omega|\geq0 \right\},\\
\mathcal{A}_2:&=\left\{( b, \lambda, \mu): b>0, \lambda\in \mathbb{R}, \mu<0, \frac{1}{3}S^{\frac{3}{2}}+\frac{\mu}{2}e^{-\frac{\lambda}{\mu}}|\Omega|>0,
\frac{1}{3}C_{b,S}+\frac{b}{12}C_{b,S}^2+\frac{\sqrt{e}}{3}\mu e^{-\frac{\lambda}{\mu}}|\Omega|>0 \right\},
\end{align*}
where $C_{b,S}:=\frac{1}{2}\left(bS^3+\sqrt{b^2S^6+4S^3}\right)$ is the positive solution of the equation $$1+bC_{b,S}-S^{-3}C_{b,S}^2=0$$
and $S>0$ is the best Sobolev embedding constant given
in \eqref{S}. Notice that $\mathcal{A}_0\supset\mathcal{A}_1\supset\mathcal{A}_2$ since
$\frac{1}{3}C_{b,S}+\frac{b}{12}C_{b,S}^2=\frac{bS^3}{12}+\frac{(4+b^2S^3)(bS^3+\sqrt{b^2S^6+4S^3})}{24}$.
The existence and multiplicity results for this case can be stated as follows.

\begin{theorem}\label{th1.1}
Suppose that $N=3$ and $(b, \lambda,\mu)\in \mathcal{A}_0$.
Then problem \eqref{eq1} has a local minimum solution with negative energy.
\end{theorem}

\begin{theorem}\label{th1.2}
Suppose that $N=3$ and $(b, \lambda,\mu)\in \mathcal{A}_1$.
Then problem \eqref{eq1} has a ground state solution with negative energy.
\end{theorem}

\begin{remark}\label{remark1.1}
In addition to the assumptions in Theorems \ref{th1.1} and \ref{th1.2}, if there holds
\begin{equation}\label{local-ground}
b^2S^3+4+\sqrt{b^4S^6+4b^2S^3}-4\sqrt{1+b|\mu|e^{\frac{1}{2}-\frac{\lambda}{\mu}}|\Omega|}\geq0,
\end{equation}
then the local minimum solution is also a ground state solution.
\end{remark}

\begin{theorem}\label{th1.3}
Suppose that $N=3$ and $(b, \lambda,\mu)\in \mathcal{A}_2$. Then problem \eqref{eq1} has a sequence of
solutions $\{u_k\}$ with $I(u_k)\leq0$ and $\|u_k\|_{H_0^1(\Omega)}\rightarrow 0$ as $k\rightarrow \infty$.
\end{theorem}

For the case $N=4$, considering the corresponding energy functional, the interaction between the Kirchhoff
term $\|u\|_{H_0^1(\Omega)}^4$ and the critical term $\|u\|_4^4$ results in some interesting phenomena.
If $bS^2-1\geq0$, the nonlocal term dominates the critical term. In this case the energy functional is
coercive in $H_0^1(\Omega)$ and satisfies the $(PS)$ condition globally. The existence of infinitely
many solutions to problem \eqref{eq1} is obtained by using the symmetric mountain pass lemma for all
$\lambda\in \mathbb{R}$ and $\mu<0$. However, if $bS^2-1<0$, the critical term plays a dominated role
and the properties of the energy functional change. The energy functional is on longer bounded from
below and fails to satisfy the $(PS)$ condition globally. In order to overcome these difficulties,
we first introduce a suitable truncation on the critical term $\|u\|_4^4$.
Then we apply the Concentration Compactness principle to obtain a local $(PS)$ condition.
Finally, with the help of the symmetric mountain pass lemma, the existence of infinitely many solutions to
problem \eqref{eq1} is obtained. Set
\begin{align*}
\mathcal{B}_0:=\left\{( b, \lambda, \mu): b>0, \lambda\in \mathbb{R}, \mu<0, \frac{S^2}{4(1-bS^2)}+
\frac{\mu}{4}e^{1-\frac{\lambda}{\mu}}|\Omega|>0 \right\}.
\end{align*}

\begin{theorem}\label{th1.4}
Suppose that $N=4$ and one of the following $(i)$ and $(ii)$ holds.\\
$(i)$ $bS^2-1\geq0$ and $\lambda\in \mathbb{R}$, $\mu<0$,\\
$(ii)$ $bS^2-1<0$ and $(b,\lambda,\mu)\in \mathcal{B}_0$.\\
Then problem \eqref{eq1} has a sequence of solutions $\{u_k\}$ with $I(u_k)\leq0$ and
$\|u_k\|_{H_0^1(\Omega)}\rightarrow 0$ as $k\rightarrow \infty$.
\end{theorem}

\begin{remark}\label{re4.1}
The existence of a local minimum solution and a ground state solution for problem \eqref{eq1}
with $N=4$ has been obtained in \cite{Zhang2024}.
\end{remark}

Lastly, we consider the case $N\geq5$. In this case, the critical Sobolev exponent is strictly
less than $4$, which, together with the fact that the sign of logarithmic term is indefinite
and that the logarithmic term does not satisfy the Aimbrosetti-Rabinowitz type condition,
makes it impossible to prove the compactness of $(PS)$ sequences. Moreover, the truncation method
can not be applied neither. Through a careful analysis of
the structure of the energy functional, we prove that problem \eqref{eq1} has a local minimum
solution. By imposing additional conditions on the parameters, we show that it is also a ground
state solution. Set
\begin{align*}
\mathcal{C}_0:&=\left\{( b, \lambda, \mu): b>0,\lambda\in\mathbb{R}, \mu<0,
 \frac{1}{2b}C_N+\frac{1}{4b}C_N^2-\frac{1}{2^*}\left(\frac{C_N}{bS}\right)^{\frac{2^*}{2}}
+\frac{\mu}{2}e^{-\frac{\lambda}{\mu}}|\Omega|>0 \right\},
\end{align*}
where  $C_N:=\frac{4}{N-4}$.

\begin{theorem}\label{th1.5}
Suppose that $N\geq 5$ and $(b,\lambda,\mu)\in \mathcal{C}_0$.  Then problem \eqref{eq1}
admits a local minimum solution with negative energy. If in addition $b\geq b^*:=\frac{2}{N-2}(\frac{N-4}{N-2})^{\frac{N-4}{2}}S^{-\frac{N}{2}}$, then the local minimum
solution is also a ground state solution.
\end{theorem}

\begin{remark}
It is seen from \cite{Naimen-3} that when $N=3$, problem \eqref{eq2} has no solution for $\lambda\leq\lambda_1/4$.
When $N\geq4$, it follows from the well-known Pohozaev's identity that problem \eqref{eq2} admits no nontrivial solutions
for $\lambda\leq 0$ if $\Omega$ is a star-shaped domain. However, Theorems \ref{th1.1}-\ref{th1.5} show that this situation can be
reversed when the logarithmic term is introduced.
\end{remark}

This paper is organized as follows. In Section $2$, we introduce some notations, definitions
and necessary lemmas which will be used in the following proofs. The cases $N=3$, $N=4$ and
$N\geq5$ are treated in Sections $3$, $4$ and $5$, respectively.

\section{Preliminaries}

We begin this section with some notations. Throughout this paper, we use $\|\cdot\|_p$ to denote
the usual $L^p(\Omega)$ norm for $1\leq p\leq\infty$ and denote the $H_0^1(\Omega)$-norm by
$\|\cdot\|:=\|\nabla \cdot\|_2$. The dual space of $H_0^1(\Omega)$ is denoted by $H^{-1}(\Omega)$
and the dual pair between $H^{-1}(\Omega)$ and $H_0^1(\Omega)$ is written as $\langle\cdot,\cdot\rangle$.
For each Banach space $B$, we use $\rightarrow$ and $\rightharpoonup$ to denote the strong and
weak convergence in it, respectively. The notation $|\Omega|$ means the Lebesgue measure of
$\Omega$ in $\mathbb{R}^N$. The symbol $O(t)$ means $|\frac{O(t)}{t}|\leq C$ as $t\rightarrow 0$
and $o_n(1)$ is an infinitesimal as $n\rightarrow\infty$. The capital letter $C$ will denote
a generic positive constant which may vary from line to line. The positive constant $S$ denotes
the best embedding constant from $H_0^1(\Omega)$ to $L^{2^{*}}(\Omega)$, i.e.,
\begin{align}\label{S}
S=\inf\limits_{u\in H_0^1(\Omega)\backslash\{0\}}\dfrac{\| u\|^2}{\|u\|_{2^{*}}^{2}}.
\end{align}

In this paper, we consider weak solutions to problem \eqref{eq1} in the following sense.
\begin{definition}\label{weaksolution}$\mathrm{\bf{(Weak \ solution)}}$
A function $u\in H_0^1(\Omega)$ is called a weak solution to problem \eqref{eq1}, if for all
$\phi\in H_0^1(\Omega)$, it holds that
$$\left(1+b\int_{\Omega}|\nabla u|^2\mathrm{d}x\right)\int_{\Omega}\nabla u\nabla \phi\mathrm{d}x
-\lambda\int_{\Omega}u \phi\mathrm{d}x -\mu\int_{\Omega}u \phi \log u^2\mathrm{d}x-\int_{\Omega}|u|^{2^{*}-2}u\phi\mathrm{d}x=0.$$
\end{definition}

The energy functional associated with problem \eqref{eq1} and its Frech\'{e}t derivative are
denoted respectively by
\begin{equation}\label{energy}
I(u)=\frac{1}{2}\|u\|^2+\frac{b}{4}\|u\|^4-\frac{\lambda}{2}\|u\|_2^2+\frac{\mu}{2}\|u\|_2^2
-\frac{\mu}{2}\int_{\Omega}u^2\log u^2\mathrm{d}x-\frac{1}{2^{*}}\|u\|_{2^*}^{2^*},\ u\in H_0^1(\Omega),
\end{equation}
and for all $u,\phi\in H_0^1(\Omega)$,
\begin{equation*}
\begin{split}
\langle I'(u),\phi\rangle=\left(1+b\|u\|^2\right)\int_{\Omega}\nabla u\nabla \phi\mathrm{d}x
-\lambda\int_{\Omega}u \phi\mathrm{d}x -\mu\int_{\Omega}u \phi \log u^2\mathrm{d}x-\int_{\Omega}|u|^{2^{*}-2}u\phi\mathrm{d}x.
\end{split}
\end{equation*}
It is well known that $I(u)$ is well defined and is a $C^1$ functional in $H_0^1(\Omega)$.
Moreover, every critical point of $I$ corresponds to a weak solution to problem \eqref{eq1}.
Since $I(u)=I(|u|)$, we may assume that $u\geq0$ in the sequel.

The following two lemmas, Br\'{e}zis-Lieb's lemma and the Concentration Compactness principle,
will play a crucial role in recovering the compactness of the approximate sequences.

\begin{lemma}\label{Brezis-Lieb}(Br\'{e}zis-Lieb's lemma \cite{BreLieb})
Let $p\in(0,\infty)$. Suppose that $\{u_n\}$ is a bounded sequence in $L^p(\Omega)$ and
$u_n\rightarrow u$ a.e. in  $\Omega$ as $n\rightarrow\infty$.
Then
\begin{eqnarray*}
&\lim\limits_{n\rightarrow\infty}(\|u_n\|_p^p-\|u_n-u\|_p^p)=\|u\|_p^p.
\end{eqnarray*}
\end{lemma}

\begin{lemma}\label{CCP}(Concentration Compactness principle \cite{Lions1,Lions2})
Let ${u_n}\subset H_0^1(\Omega)$ be a bounded sequence with weak limit $u$, $\mu$ and $\nu$
be two nonnegative and bounded measures on $\overline{\Omega}$, such that\\
$(i)$~$|\nabla u_n|^2$ converges in the weak$^{*}$ sense of measures to a measure $\mu$,\\
$(ii)$~$|u_n|^{2^{*}}$ converges in the weak$^{*}$ sense of measures to a measure $\nu$.\\
Then for some at most countable index set $\mathcal{J}$, points $\{x_k\}_{k\in \mathcal{J}}
\subset\overline{\Omega}$, values $\{\mu_k\}_{k\in \mathcal{J}}$, $\{\nu_k\}_{k\in \mathcal{J}}
\subset \mathbb{R}^{+}$, we have
\begin{equation*}
\begin{split}
(1)\ &\mu\geq |\nabla u|^2+\sum_{k\in \mathcal{J}}\mu_k \delta_{x_k}, \ \ \mu_k>0,\\
(2)\ &\nu=|u|^{2^{*}}+\sum_{k\in \mathcal{J}}\nu_k \delta_{x_k}, \ \ \nu_k>0,\\
(3)\ &\mu_k\geq S\nu_k^{\frac{2}{2^{*}}},(k\in \mathcal{J}),
\end{split}
\end{equation*}
where $\delta_{x_k}$ is the Dirac measure mass at $x_k \in \overline{\Omega}$ and $S$ is given by \eqref{S}.
\end{lemma}

Next, we introduce the Ekeland's variational principle, which will be used to show that any
minimizing sequence $\{u_n\}$ of $I(u)$ within a ball $B_r\subset H_0^1(\Omega)$ satisfies $I'(u_n)\rightarrow0$
in $H^{-1}(\Omega)$ as $n\rightarrow\infty$, where $B_{r}:=\{u\in H_0^1(\Omega): \|u\|\leq r\}$.

\begin{lemma}\label{Ekeland}(Ekeland's variational principle \cite{Ekeland})
Let $V$ be a complete metric space, and $F: V\rightarrow \mathbb{R}\cup\{+\infty\}$ be a lower
semicontinuous function, not identically $+\infty$, and bounded from below. Then for every
$\varepsilon,\delta>0$ and every point $v\in V$ such that
\begin{align*}
\inf\limits_{h\in V} F(h)\leq F(v)\leq\inf\limits_{h\in V} F(h)+\varepsilon,
\end{align*}
there exists some point $u\in V$ such that
\begin{align*}
&F(u)\leq F(v),\\
&d(u,v)\leq\delta,\\
F(w)>&F(u)-(\varepsilon/\delta) d(u,w), \ \ \forall\, w\in V, w\neq u,
\end{align*}
where $d(\cdot,\cdot)$ denotes the distance of two points in V.
\end{lemma}

To deal with the logarithmic type nonlinearity $u\log u^{2}$, we need the following basic
inequalities and the convergence properties of the logarithmic term.

\begin{lemma}\label{logarithmic inequality}
$(1)$ For all $t>0$, there holds
\begin{equation}\label{2-log1}
t\log t \geq-\frac{1}{e}.
\end{equation}

$(2)$ For each $\delta,\sigma>0$, there exists a constant $C_{\delta,\sigma}$ such that
\begin{equation}\label{2-log2}
|\log t|\leq C_{\delta,\sigma}(t^{\delta}+t^{-\sigma}), \ \ \ \ \ \forall \ t>0.
\end{equation}

$(3)$ For any $\delta>0$, there holds
\begin{equation}\label{2-log3}
\frac{\log t}{t^\delta}\leq \frac{1}{\delta e}, \qquad \forall\ t>0.
\end{equation}
\end{lemma}

\begin{lemma}\label{2-convergence}
Assume that $\{u_n\}$ is a bounded sequence in $H_0^1(\Omega)$ such that $u_n\rightarrow u$
$a.e.$ in $\Omega$ as $n\rightarrow\infty$. Then we have
\begin{align}\label{2-log-convergence1}
\lim_{n\rightarrow\infty}\int_{\Omega}u_n^{2}\log u_n^2\mathrm{d}x
=\int_{\Omega}u^{2}\log u^2\mathrm{d}x,
\end{align}
and
\begin{align}\label{2-log-convergence2}
\lim_{n\rightarrow\infty}\int_{\Omega}u_n\phi\log u_n^2\mathrm{d}x
=\int_{\Omega}u\phi\log u^2\mathrm{d}x,\qquad \forall\ \phi\in H_0^1(\Omega).
\end{align}
\end{lemma}
\begin{proof}
The proofs of \eqref{2-log-convergence1} and \eqref{2-log-convergence2} follow directly from
the strategies of $(3.4)$ and $(3.7)$ in \cite{LHW2023}.
\end{proof}

Finally, in order to illustrate the multiplicity of solutions to problem \eqref{eq1}, we
give the definition of genus and present some of its properties.

\begin{definition}\label{genus}([Section 7, \cite{1986Rabinowitz}])
Let $E$ be a real Banach space and let $\mathcal{E}$ denote the family of sets $A\subset E \backslash \{0\}$ such
that $A$ is closed in $E$ and symmetric with respect to $0$, i.e. $x \in A$ implies $-x\in A$. For $A\in \mathcal{E}$,
we define the genus of $A$ by the smallest integer $k$ (denoted by $\gamma(A)=k$) such that there exists an odd
continuous mapping from $A$ to $\mathbb{R}^k\backslash\{0\}$. If there does not exist a finite such $k$,
set $\gamma(A)=+\infty$. Finally, set $\gamma(\emptyset)=0$. Moreover, let $\mathcal{E}_k$ denote the family of
closed symmetric subsets $A$ of $E$ such that
$0 \notin A$ and $\gamma(A)\geq k$.
\end{definition}

\begin{proposition} ([Proposition 7.5, \cite{1986Rabinowitz}, Proposition 5.4 \cite{Struwe}])\label{genus-proposition}
Let $A,B\in \mathcal{E}$. Then the following properties hold.\\
$(1)$ If there is an odd continuous mapping from $A$ to $B$, then $\gamma(A)\leq \gamma(B)$.\\
$(2)$ If there is an odd homeomorphism from $A$ onto $B$, then $\gamma(A)=\gamma(B)$.\\
$(3)$ If $\gamma(B)<\infty$, then $\gamma(\overline{A\backslash B})\geq \gamma(A)-\gamma(B)$.\\
$(4)$ The $n$-dimensional sphere $S^n$ has a genus of $n+1$ by the Borsuk-Ulam theorem.\\
$(5)$ If $A$ is compact, then $\gamma(A)<\infty$ and there exists $\delta>0$ such that $N_{\delta}(A)\in \mathcal{E}$ and
$\gamma(N_{\delta}(A))=\gamma(A)$, where $N_{\delta}(A)=\{x\in E: \|x-A\|\leq\delta\}$.
\end{proposition}

The following version of the symmetric mountain pass lemma is due to Kajikiya \cite{2005Kajikiya}.

\begin{lemma}\label{symmetric mountain-pass lemma-1}
Let $E$ be an infinite-dimensional Banach space and $I\in C^1(E,\mathbb{R})$ satisfy $(A_1)$ and $(A_2)$ below.\\
$(A_1)$ $I(u)$ is even, bounded from below, $I(0)=0$ and $I(u)$ satisfies the global $(PS)$ condition, i.e. any sequence
$\{u_k\}$ in $E$ such that $\{I(u_k)\}$ is bounded and $I'(u_k)\rightarrow0$ in $E^*$ as $k\rightarrow\infty$ has a
convergent subsequence, where $E^*$ is the dual space of $E$.\\
$(A_2)$ For each $k\in \mathbb{N}$, there exists an $A_k\in \mathcal{E}_k$ such that $\sup\limits_{u\in A_k} I(u) < 0$.\\
Then either $(i)$ or $(ii)$ below holds.\\
$(i)$ There exists a sequence $\{u_k\}$ such that $I'(u_k)=0$, $I(u_k)<0$ and $\{u_k\}$ converges to zero.\\
$(ii)$ There exist two sequences $\{u_k\}$ and $\{v_k\}$ such that  $I'(u_k)=0$, $I(u_k)=0$, $u_k\neq 0$,
$\lim\limits_{k\rightarrow+\infty}u_k=0$,
$I'(v_k)=0$, $I(v_k)<0$, $\lim\limits_{k\rightarrow+\infty}I(v_k)=0$ and $\{v_k\}$ converges to a non-zero limit.
\end{lemma}

We remark that in Lemma \ref{symmetric mountain-pass lemma-1}, the functional $I(u)$ is required to
satisfy the global $(PS)$ condition. However, by a careful examination of the proof of the main
results in \cite{2005Kajikiya}, we find that it is sufficient that the functional $I(u)$ satisfies
the local $(PS)$ condition. Consequently, Kajikiya's conclusion remains true in the following form.

\begin{lemma}\label{symmetric mountain-pass lemma}
Let $E$ be an infinite-dimensional Banach space and $I\in C^1(E,\mathbb{R})$ satisfy the following conditions.\\
$(A_1)$ $I(u)$ is even, bounded from below, $I(0)=0$ and $I(u)$ satisfies the local $(PS)$ condition, i.e.
for some $c^{*}> 0$, any sequence $\{u_k\}$ in $E$ satisfying $\lim\limits_{k\rightarrow+\infty}I(u_k)=c< c^*$ and $\lim\limits_{k\rightarrow+\infty}I'(u_k)=0$ has a convergent subsequence.\\
$(A_2)$ For each $k\in \mathbb{N}$, there exists an $A_k\in \mathcal{E}_k$ such that $\sup\limits_{u\in A_k} I(u) < 0$.\\
Then either $(i)$ or $(ii)$ below holds.\\
$(i)$ There exists a sequence $\{u_k\}$ such that $I'(u_k)=0$, $I(u_k)<0$ and $\{u_k\}$ converges to zero.\\
$(ii)$ There exist two sequences $\{u_k\}$ and $\{v_k\}$ such that  $I'(u_k)=0$, $I(u_k)=0$, $u_k\neq 0$,
$\lim\limits_{k\rightarrow+\infty}u_k=0$,
$I'(v_k)=0$, $I(v_k)<0$, $\lim\limits_{k\rightarrow+\infty}I(v_k)=0$ and $\{v_k\}$ converges to a non-zero limit.
\end{lemma}

It is seen from Lemmas \ref{symmetric mountain-pass lemma-1} and \ref{symmetric mountain-pass lemma}
that in each case $(i)$ or $(ii)$, we have a sequence  of critical points $\{u_k\}$ such that $I(u_k)\leq0$,
$u_k\neq 0$ and $\lim\limits_{k\rightarrow+\infty}u_k=0$.

\par
\section{The case $N=3$}
\setcounter{equation}{0}

In this section, we consider the case $N=3$ and prove Theorems \ref{th1.1}-\ref{th1.3}. In this case, the energy
functional possesses the local minimum structure and mountain pass structure under suitable assumptions.
Combining this with Ekeland's variational principle and energy estimates, we can prove that problem \eqref{eq1}
admits a local minimum solution and a ground state solution. Furthermore, under some more assumptions on the
parameters, it is shown that the local minimum solution is also a ground state solution.

However, although the functional $I(u)$ has a mountain pass structure, it is hard to get the compactness of the
$(PS)_{c_M}$ sequence since $\mu<0$. Consequently, we don't know whether or not there exists a mountain pass type solution.

\subsection{Proof of Theorems \ref{th1.1}, \ref{th1.2} and Remark \ref{remark1.1}}

In this subsection, we give the proof of Theorems \ref{th1.1}, \ref{th1.2} and Remark \ref{remark1.1}. We first prove that
there exist two positive constants $\alpha$ and $\rho$ such that $I(u)\geq \alpha$ for all $u\in H_0^1(\Omega)$ with $\|u\|=\rho$.

\begin{lemma}\label{le3.1}
Suppose that $(b, \lambda,\mu)\in \mathcal{A}_0$.
Then there exist constants $\alpha,\rho>0$ such that $I(u)\geq \alpha$ for all $u\in H_0^1(\Omega)$ with $\|u\|=\rho$.
\end{lemma}

\begin{proof}
For all $u\in H_0^1(\Omega)\backslash \{0\}$, since $\mu<0$, by using inequality \eqref{2-log1}, we have
\begin{align}\label{log-3-1}
&-\frac{\lambda}{2}\|u\|_{2}^{2}+\frac{\mu}{2}\|u\|_{2}^{2}-\frac{\mu}{2}\int_{\Omega}u^{2}\log u^{2}\mathrm{d}x\nonumber\\
=&-\frac{\mu}{2}\int_{\Omega}u^{2}\left(\frac{\lambda}{\mu}-1+\log u^{2}\right)\mathrm{d}x\\
=&-\frac{\mu}{2}e^{1-\frac{\lambda}{\mu}}\int_{\Omega}e^{\frac{\lambda}{\mu}-1}u^{2}\log \left(e^{\frac{\lambda}{\mu}-1}u^{2}\right)\mathrm{d}x\nonumber\\
\geq& \ \frac{\mu}{2}e^{-\frac{\lambda}{\mu}}|\Omega|.\nonumber
\end{align}
Then it follows from \eqref{S} and \eqref{log-3-1} that
\begin{align}\label{3-ineq-1}
I(u)&=\frac{1}{2}\|u\|^2+\frac{b}{4}\|u\|^4-\frac{\lambda}{2}\|u\|_2^2+\frac{\mu}{2}\|u\|_2^2
-\frac{\mu}{2}\int_{\Omega}u^2\log{u^2}\mathrm{d}x-\frac{1}{6}\|u\|_{6}^{6}\nonumber\\
&\geq \frac{1}{2}\|u\|^2+\frac{b}{4}\|u\|^4+\frac{\mu}{2}e^{-\frac{\lambda}{\mu}}|\Omega|-\frac{1}{6S^3}\|u\|^{6}.
\end{align}
Set
\begin{align*}
f(t)=\frac{1}{2}t^2+\frac{b}{4}t^4+\frac{\mu}{2}e^{-\frac{\lambda}{\mu}}|\Omega|-\frac{1}{6S^{3}}t^{6}, \ \ \ t>0,
\end{align*}
$\rho:=\left(\frac{bS^3+\sqrt{b^2S^6+4S^3}}{2}\right)^{1/2}$ and $\alpha:=\dfrac{bS^3}{12}+\dfrac{(4+b^2S^3)(bS^3+\sqrt{b^2S^6+4S^3})}{24}+\dfrac{\mu}{2}e^{-\frac{\lambda}{\mu}}|\Omega|$.
Since $(b, \lambda,\mu)\in \mathcal{A}_0$, by a direct calculation we have $I(u)\geq f(\rho)=\alpha>0$ for all
$\|u\|=\rho$. This completes the proof.
\end{proof}

Define
$$c_{\rho}:=\inf_{u\in B_{\rho}}I(u), \ \ \ \ \ \ \ \ \ \ \ \ \ \ \ \ \ c_{\mathcal{K}}:=\inf_{u\in \mathcal{K}}I(u),$$
where $\rho>0$ is given in the proof of Lemma \ref{le3.1}, $B_{\rho}:=\{u\in H_0^1(\Omega): \|u\|\leq \rho\}$ and
$\mathcal{K}=\{u\in H_0^1(\Omega)\backslash \{0\}:I'(u)=0\}$. Notice that $\widetilde{u}\in\mathcal{K}$,
where $\widetilde{u}$ is the local minimum solution obtained in Theorem \ref{th1.1}. Hence $\mathcal{K}$ is
nonempty once Theorem \ref{th1.1} is proved.

Next, we show that both $c_{\rho}$ and $c_{\mathcal{K}}$ are well defined and negative.

\begin{lemma}\label{le3.2}
Suppose that $b>0$, $\lambda\in \mathbb{R}$ and $\mu<0$. Then we have
$$-\infty <c_{\mathcal{K}},c_{\rho}<0.$$
\end{lemma}

\begin{proof}
(I) We first prove $-\infty<c_\rho<0$. The left part follows easily from inequality \eqref{3-ineq-1}
and the definition of $c_{\rho}$. To show that $c_{\rho}<0$, for any $u\in H_0^1(\Omega)\setminus\{0\}$,
consider $I(tu)$ with $t>0$, i.e.,
\begin{align*}
I(tu)=t^2\left(\frac{1}{2}\| u\|^{2}+\frac{b}{4}t^2\|u\|^4-\frac{\lambda}{2}\|u\|_{2}^{2}-\frac{\mu}{2}\log t^2\|u\|_{2}^{2}
+\frac{\mu}{2}\int_{\Omega}u^{2}(1-\log u^{2})\mathrm{d}x-\frac{1}{6}t^{4}\|u\|_{6}^{6}\right).
\end{align*}
Since $\mu<0$, we deduce that there exists a $t_u>0$ suitably small such that $\|t_u u\|\leq \rho$ and $I(t_uu)<0$,
which implies that $c_{\rho}<0$.

(II) Next we show that $-\infty<c_{\mathcal{K}}<0$. We need only to prove $c_{\mathcal{K}}>-\infty$,
since $c_\mathcal{K}\leq c_\rho<0$. Noticing that $\mu<0$, by an argument similar to the proof of \eqref{log-3-1}, we have
\begin{align}\label{ineq-3-1}
-\frac{\lambda}{3}\|u\|_{2}^{2} +\frac{\mu}{2}\|u\|_{2}^{2}-\frac{\mu}{3}\int_{\Omega}u^{2}\log u^{2}\mathrm{d}x
\geq\frac{\mu}{3}e^{\frac{1}{2}-\frac{\lambda}{\mu}}|\Omega|.
\end{align}
Consequently, for any $u\in\mathcal{K}$, it follows from \eqref{ineq-3-1} that
\begin{align}\label{3-ineq-11}
I(u)&=I(u)-\frac{1}{6}\langle I'(u),u\rangle\nonumber\\
&=\frac{1}{3}\|u\|^{2}+\frac{b}{12}\|u\|^{4}-\frac{\lambda}{3}\|u\|_{2}^{2} +\frac{\mu}{2}\|u\|_{2}^{2}
-\frac{\mu}{3}\int_{\Omega}u^{2}\log u^{2}\mathrm{d}x\nonumber\\
&\geq-\frac{\lambda}{3}\|u\|_{2}^{2} +\frac{\mu}{2}\|u\|_{2}^{2}-\frac{\mu}{3}\int_{\Omega}u^{2}\log u^{2}\mathrm{d}x\\
&\geq\frac{\mu}{3}e^{\frac{1}{2}-\frac{\lambda}{\mu}}|\Omega|. \nonumber
\end{align}
Thus $c_{\mathcal{K}}>-\infty$. The proof is complete.
\end{proof}

Based on Lemma \ref{le3.1} and the energy estimate $-\infty <c_{\rho}<0$, we can apply the Eklend's variational principle
and Br\'{e}zis-Lieb's lemma to prove the compactness of the minimizing sequence for $c_\rho$.

\begin{lemma}\label{le3.3}
Assume that $(b, \lambda,\mu)\in \mathcal{A}_0$.
Then there exists a minimizing sequence $\{u_n\}$ for $c_\rho$ and a $\widetilde{u}\in H_0^1(\Omega)$ such that,
up to a subsequence, $u_n \rightarrow \widetilde{u}$ in $H_0^1(\Omega)$ as $n\rightarrow\infty$.
\end{lemma}

\begin{proof} Let $\{v_n\}\subset B_\rho$ be a minimizing sequence for $c_\rho$. Without
loss of generality, we may assume that $c_\rho\leq I(v_n)\leq c_\rho+\dfrac{1}{n}$ for all
$n\in\mathbb{N}$. According to \eqref{3-ineq-1} and the energy estimate $-\infty <c_{\rho}<0$,
one knows that there is a constant $\tau>0$ suitably small such
that $\|v_n\|\leq\rho-\tau$ for all $n$ large enough. Then using Lemma \ref{Ekeland},
we deduce that there exists another minimizing sequence $\{u_n\}\subset B_\rho$ for $c_\rho$ such that
\begin{align}\label{3-Ekeland}
&I(u_n)\leq I(v_n),\nonumber\\
&\|u_n-v_n\|\leq\frac{\tau}{2},\nonumber\\
I(w)>&I(u_n)-\frac{2}{n\tau} \|w-u_n\|,\qquad\forall\ w\in B_\rho,\ w\neq u_n.
\end{align}
It is easy to see that $\|u_n\|\leq \rho-\frac{\tau}{2}$. Now, we prove that $I'(u_n)\rightarrow 0$
as $n\rightarrow \infty$. For any $h\in H_0^1(\Omega)$ with $\|h\|=1$, set $z(t)=u_n+th$, $t\in(0,\frac{\tau}{2}]$.
Then $\|z(t)\|\leq \rho,\ \forall\ t\in(0,\frac{\tau}{2}]$. Taking $w=z(t)$ in \eqref{3-Ekeland} shows that
\begin{align}\label{3-Ekeland-ineq-1}
\frac{I(z(t))-I(u_n)}{t}>-\frac{2}{n\tau}.
\end{align}
Letting $t\rightarrow0^+$ in \eqref{3-Ekeland-ineq-1}, we arrive at
\begin{align}\label{3-Ekeland-ineq-2}
\langle I'(u_n), h \rangle\geq-\frac{2}{n\tau}.
\end{align}
Replacing $h$ by $-h$ in \eqref{3-Ekeland-ineq-2}, one has
\begin{align*}
\langle I'(u_n), h \rangle\leq\frac{2}{n\tau},
\end{align*}
which, together with \eqref{3-Ekeland-ineq-2} and the arbitrariness of $h$, implies $I'(u_n)\rightarrow 0$ as $n\rightarrow \infty$.

Since $\{u_n\}$ is bounded in $H_0^1(\Omega)$,  we may assume that there exist a subsequence
of $\{u_n\}$ (still denoted by $\{u_n\}$) and a $\widetilde{u}\in H_0^1(\Omega)$ such that,
as $n\rightarrow\infty$,
\begin{equation}\label{3-convergence}
\begin{cases}
u_{n}\rightharpoonup \widetilde{u} \ in \ H_{0}^1(\Omega),\\
u_{n}\rightarrow \widetilde{u} \ in \ L^p(\Omega), \ \ 1\leq p<6,\\
u_n^5\rightharpoonup \widetilde{u}^5 \ in \ L^{\frac{6}{5}}(\Omega),\\
u_{n}\rightarrow \widetilde{u} \ \ a.e.\ in \ \Omega.
\end{cases}
\end{equation}
Set $w_n=u_n-\widetilde{u}$. Then $\{w_n\}$ is also bounded in $H_0^1(\Omega)$. Hence, we may assume
without loss of generality that
\begin{align}\label{3-wn-1}
\lim\limits_{n\rightarrow\infty}\|w_n\|^2=l\geq0.
\end{align}

We claim $l=0$. Assume on the contrary that $l>0$. Since $u_{n}\rightharpoonup \widetilde{u}$
in $H_{0}^1(\Omega)$ as $n\rightarrow\infty$, we have
\begin{align}\label{3-wn-2}
\|u_n\|^{2}=\| w_n\|^2+\| \widetilde{u}\|^{2}+o_n(1), \qquad  n\rightarrow\infty.
\end{align}
It then follows from \eqref{2-log-convergence2}, \eqref{3-convergence}, \eqref{3-wn-2} and $I'(u_n)\rightarrow0$
as $n\rightarrow\infty$ that
\begin{align}\label{3-Iun-u}
o_n(1)=&\langle I'(u_n),\widetilde{u}\rangle\nonumber\\
=&\left(1+b\|u_n\|^2\right)\int_{\Omega}\nabla u_n\nabla \widetilde{u}\mathrm{d}x-\lambda\int_{\Omega}u_n \widetilde{u}\mathrm{d}x
-\mu\int_{\Omega}u_n \widetilde{u}\log u_n^2\mathrm{d}x-\int_{\Omega}u_n^5 \widetilde{u}\mathrm{d}x\nonumber\\
=&\|\widetilde{u}\|^2+b\|\widetilde{u}\|^4+b\|w_n\|^2\|\widetilde{u}\|^2-\lambda\|\widetilde{u}\|_2^2
-\mu\int_{\Omega}\widetilde{u}^2\log \widetilde{u}^2\mathrm{d}x-\|\widetilde{u}\|_6^6+o_n(1).
\end{align}
On the other hand, by recalling $I'(u_n)\rightarrow0$ as $n\rightarrow\infty$, \eqref{2-log-convergence1},  \eqref{3-convergence},
\eqref{3-wn-2}, \eqref{3-Iun-u} and Lemma \ref{Brezis-Lieb}, we get, as $n\rightarrow\infty$,
\begin{align}\label{3-Iun-un}
o_n(1)=&\langle I'(u_n),u_n\rangle\nonumber\\
=&\left(1+b\|u_n\|^2\right)\|u_n\|^2-\lambda \|u_n\|_2^2-\mu\int_{\Omega}u_n^2\log u_n^2\mathrm{d}x-\|u_n\|_6^6\nonumber\\
=&\|w_n\|^2+\|\widetilde{u}\|^2+b\|w_n\|^4+2b\|w_n\|^2\|\widetilde{u}\|^2+b\|\widetilde{u}\|^4-\lambda\|\widetilde{u}\|_2^2\nonumber\\
&-\mu\int_{\Omega}\widetilde{u}^2\log \widetilde{u}^2\mathrm{d}x-\|w_n\|_6^6-\|\widetilde{u}\|_6^6+o_n(1)\nonumber\\
=&\|w_n\|^2+b\|w_n\|^4+b\|w_n\|^2\|\widetilde{u}\|^2-\|w_n\|_6^6+o_n(1).
\end{align}
In view of \eqref{2-log-convergence1}, \eqref{3-convergence}, \eqref{3-wn-2}, \eqref{3-Iun-un} and Lemma \ref{Brezis-Lieb},
one obtains
\begin{align}\label{3-Iun=c}
c_{\rho}+ o_n(1)=I(u_n)=&\frac{1}{2}\|u_n\|^2+\frac{b}{4}\|u_n\|^4-\frac{\lambda}{2}\|u_n\|_2^2+\frac{\mu}{2}\|u_n\|_2^2
-\frac{\mu}{2}\int_{\Omega}u_n^2\log{u_n^2}\mathrm{d}x-\frac{1}{6}\|u_n\|_{6}^{6}\nonumber\\
=&\frac{1}{2}\|w_n\|^2+\frac{1}{2}\|\widetilde{u}\|^2+\frac{b}{4}\|w_n\|^4+\frac{b}{4}\|\widetilde{u}\|^4+\frac{b}{2}\|w_n\|^2\|\widetilde{u}\|^2
-\frac{\lambda}{2}\|\widetilde{u}\|_2^2\nonumber\\
&+\frac{\mu}{2}\|\widetilde{u}\|_2^2-\frac{\mu}{2}\int_{\Omega}\widetilde{u}^2\log{\widetilde{u}^2}\mathrm{d}x-\frac{1}{6}\|w_n\|_{6}^{6}
-\frac{1}{6}\|\widetilde{u}\|_{6}^{6}+o_n(1)\nonumber\\
=&I(\widetilde{u})+\frac{1}{2}\|w_n\|^2+\frac{b}{4}\|w_n\|^4+\frac{b}{2}\|w_n\|^2\|\widetilde{u}\|^2-\frac{1}{6}\|w_n\|_{6}^{6}
+o_n(1)\nonumber\\
=&I(\widetilde{u})+\frac{1}{6}\langle I'(u_n),u_n\rangle+\frac{1}{3}\|w_n\|^2
+\frac{b}{3}\|w_n\|^2\|\widetilde{u}\|^2+\frac{b}{12}\|w_n\|^4+o_n(1)\nonumber\\
=&I(\widetilde{u})+\frac{1}{3}\|w_n\|^2+\frac{b}{3}\|w_n\|^2\|\widetilde{u}\|^2+\frac{b}{12}\|w_n\|^4+o_n(1).
\end{align}
Letting $n\rightarrow\infty$ in \eqref{3-Iun=c} and noting that $b,l>0$, we have
\begin{align}\label{3-inequation}
c_{\rho}=I(\widetilde{u})+\frac{1}{3}l+\frac{b}{3}l\|\widetilde{u}\|^2+\frac{b}{12}l^2>I(\widetilde{u}).
\end{align}
However, by the weak lower semi-continuity of the norm, we have $\|\widetilde{u}\|\leq\rho$.
Consequently, $I(\widetilde{u})\geq c_{\rho}$, which contradicts with \eqref{3-inequation}.
Thus $l=0$, i.e., $u_n\rightarrow \widetilde{u}$ in $H_0^1(\Omega)$ as $n\rightarrow\infty$. The proof is complete.
\end{proof}

\begin{lemma}\label{le3.4}
Assume that $(b, \lambda,\mu)\in \mathcal{A}_1$.
Then there exists a minimizing sequence $\{u_n\}\subset\mathcal{K}$ for $c_{\mathcal{K}}$ and a $u_g\in H_0^1(\Omega)$ such that,
up to a subsequence, $u_n \rightarrow u_g$ in $H_0^1(\Omega)$ as $n\rightarrow\infty$.
\end{lemma}

\begin{proof}
Let $\{u_n\}\subset\mathcal{K}$ be a minimizing sequence for $c_{\mathcal{K}}$. We first show that $\{u_n\}$ is bounded.
Indeed, from $I(u_n)\rightarrow c_\mathcal{K}$ as $n\rightarrow \infty$ and \eqref{ineq-3-1},
one obtains, for $n$ suitably large, that
\begin{align*}
c_\mathcal{K}+1&\geq I(u_n)-\frac{1}{6}\langle I'(u_n),u_n\rangle\\
&=\frac{1}{3}\|u_n\|^{2}+\frac{b}{12}\|u_n\|^{4}-\frac{\lambda}{3}\|u_n\|_{2}^{2} +\frac{\mu}{2}\|u_n\|_{2}^{2}
-\frac{\mu}{3}\int_{\Omega}u_n^{2}\log u_n^{2}\mathrm{d}x\\
&\geq\frac{1}{3}\|u_n\|^{2}+\frac{b}{12}\|u_n\|^{4}+\frac{\mu}{3}e^{\frac{1}{2}-\frac{\lambda}{\mu}}|\Omega|,
\end{align*}
which ensures the boundedness of $\{u_n\}$. Consequently, there exist a subsequence of $\{u_n\}$ (still
denoted by $\{u_n\}$) and a $u_g\in H_0^1(\Omega)$ such that, as $n\rightarrow\infty$,
\begin{equation}\label{3-1-un}
\begin{cases}
u_n&\rightharpoonup \ u_g \ in \ H_0^1(\Omega),\\
u_n&\rightarrow \ u_g \ in \ L^p(\Omega), \ \ \  1\leq p<6,\\
u_n^5&\rightharpoonup u_g^5 \ in \ L^{\frac{6}{5}}(\Omega),\\
u_n&\rightarrow \ u_g \ a.e. \ in \ \Omega.
\end{cases}
\end{equation}
According to Lemma \ref{CCP}, there exist an at most countable index set $\mathcal{J}$ and a collection of points
$\{x_k\}_{k\in \mathcal{J}}$ such that
\allowdisplaybreaks \begin{align*}
&|\nabla u_n|^2\rightharpoonup\mu\geq |\nabla u_g|^2+\sum_{k\in \mathcal{J}}\mu_k \delta_{x_k}, \ \ \mu_k>0,\nonumber\\
&|u_n|^{6}\rightharpoonup \nu=|u_g|^{6}+\sum_{k\in \mathcal{J}}\nu_k \delta_{x_k}, \ \ \nu_k>0,
\end{align*}
in the measure sense.
In addition, we also have
\begin{equation}\label{3-1-uv}
\mu_k \geq S\nu_k^{\frac{1}{3}}, \ \ \ \ k\in \mathcal{J}.
\end{equation}

Now we claim that $\mathcal{J}=\emptyset$. Suppose on the contrary that $\mathcal{J}\neq\emptyset$ and fix $k\in \mathcal{J}$.
Define a function $\phi\in C_0^\infty(\Omega)$ with $\phi(x)=1$ in $B(x_k,\rho)$, $\phi(x)=0$ in
$\Omega\backslash B(x_k,2\rho)$, $0\leq \phi(x)\leq1$ otherwise, and with $|\nabla\phi|\leq\frac{2}{\rho}$ in $\Omega$,
where $\rho >0$ is a constant.
Noticing that $I'(u_n)=0$  and $\{u_n\phi\}$ is bounded in $H_0^1(\Omega)$, we have
\begin{align}\label{3-11-eq1}
0=&\lim\limits_{n\rightarrow\infty}\langle I'(u_n),u_n\phi\rangle\nonumber\\
=&\lim\limits_{n\rightarrow\infty}\left\{\left(1+b\int_{\Omega}|\nabla u_n|^2\mathrm{d}x\right)\int_{\Omega}\nabla {u_n}\nabla (u_n \phi)\mathrm{d}x-\lambda\int_{\Omega}u_n^2 \phi\mathrm{d}x-\mu\int_{\Omega}u_n^2 \phi\log u_n^2\mathrm{d}x\right\}\nonumber\\
&-\lim\limits_{n\rightarrow\infty}\int_{\Omega} u_n^6 \phi\mathrm{d}x\nonumber\\
=&\lim\limits_{n\rightarrow\infty}\left\{\left(1+b\int_{\Omega}|\nabla u_n|^2\mathrm{d}x\right)\int_{\Omega}|\nabla {u_n}|^2\phi\mathrm{d}x
+\left(1+b\int_{\Omega}|\nabla u_n|^2\mathrm{d}x\right)\int_{\Omega}\nabla {u_n}u_n\nabla \phi\mathrm{d}x\right\}\nonumber\\
&-\lim\limits_{n\rightarrow\infty}\left\{\lambda\int_{\Omega}u_n^2 \phi\mathrm{d}x+\mu\int_{\Omega}u_n^2 \phi\log u_n^2\mathrm{d}x
+\int_{\Omega} u_n^6 \phi\mathrm{d}x\right\}\nonumber\\
\geq&\lim\limits_{n\rightarrow\infty}\left\{\left(1+b\int_{\Omega}|\nabla u_n|^2\phi\mathrm{d}x\right)
\int_{\Omega}|\nabla {u_n}|^2\phi\mathrm{d}x
-\int_{\Omega} u_n^6 \phi\mathrm{d}x\right\}+o(1),
\end{align}
where $o(1)\rightarrow0$ as $\rho\rightarrow0$. The last inequality follows from the fact,  as $\rho\rightarrow0$,
\begin{align}\label{3-11-equation1}
\lim\limits_{n\rightarrow\infty}\left\{\left(1+b\int_{\Omega}|\nabla u_n|^2\mathrm{d}x\right)\int_{\Omega}\nabla {u_n}u_n\nabla \phi\mathrm{d}x
-\lambda\int_{\Omega}u_n^2 \phi\mathrm{d}x-\mu\int_{\Omega}u_n^2 \phi\log u_n^2\mathrm{d}x\right\}=o(1).
\end{align}
Next, we prove \eqref{3-11-equation1}. Since $\{u_n\}$ is bounded in $H_0^1(\Omega)$, by means of \eqref{3-1-un},
H\"{o}lder's inequality and the definition of $\phi$, one has
\begin{eqnarray}\label{3-11-equation2}
&&\limsup\limits_{n\rightarrow\infty}\bigg|\left(1+b\int_{\Omega}|\nabla u_n|^2\mathrm{d}x\right)\int_{\Omega}\nabla {u_n}u_n\nabla \phi\mathrm{d}x\bigg|\nonumber\\
&\leq&\limsup\limits_{n\rightarrow\infty}C\int_{\Omega\cap B(x_k,2\rho)}|\nabla u_n u_n\nabla\phi|\mathrm{d}x\nonumber\\
&\leq&\limsup\limits_{n\rightarrow\infty}C\left(\int_{\Omega\cap B(x_k,2\rho)}|\nabla u_n|^2\mathrm{d}x\right)^{\frac{1}{2}}
\left(\int_{\Omega\cap B(x_k,2\rho)}|u_n \nabla\phi|^2\mathrm{d}x\right)^{\frac{1}{2}}\nonumber\\
&\leq&\limsup\limits_{n\rightarrow\infty}C\left(\int_{\Omega\cap B(x_k,2\rho)}|u_n \nabla\phi|^2\mathrm{d}x\right)^{\frac{1}{2}}\nonumber\\
&=&C\left(\int_{\Omega\cap B(x_k,2\rho)}|u_g \nabla\phi|^2\mathrm{d}x\right)^{\frac{1}{2}}\nonumber\\
&\leq&C\left(\int_{\Omega\cap B(x_k,2\rho)}|u_g|^6\mathrm{d}x\right)^{\frac{1}{6}}
\left(\int_{\Omega\cap B(x_k,2\rho)}|\nabla\phi|^3\mathrm{d}x\right)^{\frac{1}{3}}\nonumber\\
&\leq&C\left(\int_{\Omega\cap B(x_k,2\rho)}|u_g|^6\mathrm{d}x\right)^{\frac{1}{6}} \ \ \ \rightarrow  \ 0, \ \ as \ \rho\rightarrow 0.
\end{eqnarray}
Using \eqref{3-1-un} and the definition of $\phi$ again, we deduce that
\begin{eqnarray}\label{3-11-equation3}
&&\limsup\limits_{n\rightarrow\infty}\bigg|\int_{\Omega}u_n^2 \phi\mathrm{d}x\bigg|\nonumber\\
&=&\int_{\Omega\cap B(x_k,2\rho)}u_g^2 \phi\mathrm{d}x\nonumber\\
&\leq&\int_{\Omega\cap B(x_k,2\rho)}u_g^2 \mathrm{d}x \ \ \ \rightarrow  \ 0, \ \ as \ \rho\rightarrow 0.
\end{eqnarray}
Recalling \eqref{2-log2} with $\delta=\sigma=\frac{1}{2}$, \eqref{2-log-convergence1} and the definition of $\phi$,
one obtains
\begin{eqnarray}\label{3-11-equation4}
&&\limsup\limits_{n\rightarrow\infty}\bigg|\int_{\Omega}u_n^2 \phi\log u_n^2\mathrm{d}x\bigg|\nonumber\\
&=&\bigg|\int_{\Omega\cap B(x_k,2\rho)}u_g^2 \phi\log u_g^2\mathrm{d}x\bigg|\nonumber\\
&\leq&C\int_{\Omega\cap B(x_k,2\rho)}(u_g^3+u_g) \mathrm{d}x \ \ \ \rightarrow  \ 0, \ \ as \ \rho\rightarrow 0.
\end{eqnarray}
It follows from \eqref{3-11-equation2}, \eqref{3-11-equation3} and \eqref{3-11-equation4}  that \eqref{3-11-equation1} holds.

Letting  $\rho\rightarrow 0$ in \eqref{3-11-eq1}, we have
$$0\geq(1+b\mu_k)\mu_k-\nu_k.$$
Combing this with \eqref{3-1-uv}, we have
\begin{align}\label{3-11-Ck}
\mu_k\geq C_{b,S}:=\frac{1}{2}\left(bS^3+\sqrt{b^2S^6+4S^3}\right).
\end{align}
Since $I(u_n)\rightarrow c_{\mathcal{K}}$ as $n\rightarrow\infty$ and $I'(u_n)=0$,  by using the boundedness of $\{u_n\}$ and
\eqref{ineq-3-1}, one has
\begin{align}\label{3-11-ineq2}
c_{\mathcal{K}}=&\lim\limits_{n\rightarrow\infty}\left\{I(u_n)-\frac{1}{6}\langle I'(u_n),u_n\rangle\right\}\nonumber\\
=&\lim\limits_{n\rightarrow\infty}\bigg\{\frac{1}{3}\|u_n\|^2+\frac{b}{12}\|u_n\|^4-\frac{\lambda}{3}\|u_n\|_2^2+\frac{\mu}{2}\|u_n\|_2^2
-\frac{\mu}{3}\int_{\Omega}u_n^2\log{u_n^2}\mathrm{d}x\bigg\}\nonumber\\
\geq&\lim\limits_{n\rightarrow\infty}\left\{\frac{1}{3}\int_{\Omega}|\nabla u_n|^2\phi\mathrm{d}x
+\frac{b}{12}\left(\int_{\Omega}|\nabla u_n|^2\phi\mathrm{d}x\right)^2+\frac{\mu}{3}e^{\frac{1}{2}-\frac{\lambda}{\mu}}|\Omega|\right\}.
\end{align}
Letting $\rho\rightarrow0$ in \eqref{3-11-ineq2}, since $(b, \lambda,\mu)\in \mathcal{A}_1$, by \eqref{3-11-Ck} we have
\begin{align*}
c_{\mathcal{K}}\geq&\frac{1}{3}\mu_k+\frac{b}{12}\mu_k^2+\frac{\mu}{3}e^{\frac{1}{2}-\frac{\lambda}{\mu}}|\Omega|\\
\geq&\frac{1}{3}C_{b,S}+\frac{b}{12}C_{b,S}^2+\frac{\mu}{3}e^{\frac{1}{2}-\frac{\lambda}{\mu}}|\Omega|\\
\geq&0.
\end{align*}
which contradicts with $c_{\mathcal{K}}<0$. Consequently, $\mathcal{J}=\emptyset$ and
\begin{equation}\label{3-11-equa1}
\int_{\Omega}|u_n|^{6}\mathrm{d}x\rightarrow\int_{\Omega}|u_g|^{6}\mathrm{d}x, \ \ as \ \ n\rightarrow \infty.
\end{equation}

Finally, by combining $I'(u_n)=0$  with \eqref{3-1-un} and \eqref{3-11-equa1} and
applying Lemma \ref{2-convergence}, we have, for $n$ large enough,
\begin{align}\label{3-11-un-u}
0=&\langle I'(u_n),u_n-u_g\rangle\nonumber\\
=&(1+b\|u_n\|^2)\int_{\Omega}\nabla u_n\nabla(u_n-u_g)\mathrm{d}x-\lambda\int_{\Omega}u_n(u_n-u_g)\mathrm{d}x
-\mu\int_{\Omega}u_n(u_n-u_g)\log u_n^2\mathrm{d}x\nonumber\\
&-\int_{\Omega} u_n^5(u_n-u_g)\mathrm{d}x\nonumber\\
=&(1+b\|u_n\|^2)\int_{\Omega}\nabla u_n\nabla(u_n-u_g)\mathrm{d}x+o_n(1).
\end{align}
From \eqref{3-11-un-u} we have
$$\int_{\Omega}|\nabla u_n|^2\mathrm{d}x-\int_{\Omega}|\nabla u_g|^2\mathrm{d}x=o(1), \ \ as \ n\rightarrow\infty,$$
which, together with $u_n\rightharpoonup u_g \ in \ H_0^1(\Omega)$, implies that $u_n\rightarrow u_g \ in \ H_0^1(\Omega)$
as $n\rightarrow\infty$. This completes the proof.
\end{proof}

{\bf The Proof of Theorems \ref{th1.1}, \ref{th1.2} and Remark \ref{remark1.1}:}
Assume that $N=3$ and $(b, \lambda,\mu)\in \mathcal{A}_0$. By Lemmas \ref{le3.2} and \ref{le3.3} and
the fact that $I(u)$ is a $C^1$ functional in $H_0^1(\Omega)$, we know that $\widetilde{u}$ is a local minimum
solution to problem \eqref{eq1} with $I(\widetilde{u})=c_{\rho}<0$. Similarly, when $N=3$ and $(b, \lambda,\mu)\in \mathcal{A}_1$,
from Lemmas \ref{le3.2} and \ref{le3.4} and the fact that $I(u)$ is a $C^1$ functional in $H_0^1(\Omega)$,
we can deduce that $u_g$ is a ground state solution to
problem \eqref{eq1} with $I(u_g)=c_{\mathcal{K}}<0$. Theorems \ref{th1.1} and \ref{th1.2} are proved.

To prove Remark \ref{remark1.1}, it suffices to show that $c_{\rho}=c_{\mathcal{K}}$ when both the assumptions
in Theorems \ref{th1.1}-\ref{th1.2} and \eqref{local-ground} hold. Obviously, $c_{\rho}\geq c_{\mathcal{K}}$. It remains
to prove the reversed inequality. Since $u_g$ is a ground state solution to problem \eqref{eq1} with
$I(u_g)=c_{\mathcal{K}}<0$, one obtains
\begin{align*}
0>c_\mathcal{K}=I(u_g)-\frac{1}{6}\langle I'(u_g),u_g\rangle
\geq\frac{1}{3}\|u_g\|^{2}+\frac{b}{12}\|u_g\|^{4}+\frac{\mu}{3}e^{\frac{1}{2}-\frac{\lambda}{\mu}}|\Omega|,
\end{align*}
which implies
\begin{equation}\label{local-ground-2}
\|u_g\|^2< \frac{-2+2\sqrt{1+b|\mu|e^{\frac{1}{2}-\frac{\lambda}{\mu}}|\Omega|}}{b}\leq \rho^2.
\end{equation}
Here we have used \eqref{local-ground} and the choice of $\rho$ to obtain the second inequality in \eqref{local-ground-2}.
Then the fact $c_{\rho}\leq c_{\mathcal{K}}$ follows directly from \eqref{local-ground-2}. The proof of
Remark \ref{remark1.1} is complete.

\subsection{Proof of Theorem \ref{th1.3}}

In this subsection, we prove Theorem \ref{th1.3}. Since $2^*=6>4$ when $N=3$, the energy functional $I(u)$
is not bounded from below in $H_0^1(\Omega)$ and Lemma \ref{symmetric mountain-pass lemma} can not be applied directly.
To avoid this difficulty, we introduce an appropriate truncation on the critical term. Then by using Lemma
\ref{symmetric mountain-pass lemma} and the Concentration Compactness principle, we prove that problem \eqref{eq1} possesses
a sequence of solutions $\{u_k\}$ whose energies and $H_0^1(\Omega)$-norms converge to $0$ as $k\rightarrow\infty$.

To begin with, for each $u\in H_0^1(\Omega)$, one gets, by using \eqref{S} and \eqref{log-3-1}, that
\begin{align*}
I(u)&=\frac{1}{2}\|u\|^2+\frac{b}{4}\|u\|^4-\frac{\lambda}{2}\|u\|_2^2+\frac{\mu}{2}\|u\|_2^2
-\frac{\mu}{2}\int_{\Omega}u^2\log{u^2}\mathrm{d}x-\frac{1}{6}\|u\|_{6}^{6}\nonumber\\
&\geq \frac{1}{2}\|u\|^2+\frac{\mu}{2}e^{-\frac{\lambda}{\mu}}|\Omega|-\frac{1}{6S^3}\|u\|^{6}.
\end{align*}
Set
\begin{equation}\label{G}
G(t)=\frac{1}{2}t^2+\frac{\mu}{2}e^{-\frac{\lambda}{\mu}}|\Omega|-\frac{1}{6S^{3}}t^{6}, \ \ \ t\geq0.
\end{equation}
It is directly checked that $G'(t)>0$ for $0<t<t_{max}$ and $G'(t)<0$ for $t>t_{max}$, where $t_{max}:=S^{\frac{3}{4}}$.
Therefore, $G(t)$ attains its maximum at $t_{max}$ and
\begin{eqnarray*}
G_{max}:=G(t_{max})=\frac{1}{3}S^{\frac{3}{2}}+\frac{\mu}{2}e^{-\frac{\lambda}{\mu}}|\Omega|.
\end{eqnarray*}
Since $(b,\lambda,\mu)\in \mathcal{A}_2$, $G_{max}>0$.
Set $G_0=\frac{1}{2}G_{max}=\frac{1}{6}S^{\frac{3}{2}}+\frac{\mu}{4}e^{-\frac{\lambda}{\mu}}|\Omega|$,
then $G_0>0$ and there exists a unique $t_0\in(0,t_{max})$ such that $G(t_0)=G_0$.

Define a truncation function $\psi(x)\in C^1[0,\infty)$ such that
\begin{equation*}
\psi(x)=\begin{cases}
1, \ \ \ \ \  &0\leq x \leq t_0^2,\\
0\leq\psi(x)\leq1,&t_0^2\leq x \leq t_{max}^2,\\
3S^3\dfrac{1}{x^2}+\left(\frac{\mu}{2}e^{-\frac{\lambda}{\mu}}|\Omega|-G_{max}\right)6S^3\dfrac{1}{x^3}, &x\in [t_{max}^2,+\infty).
\end{cases}
\end{equation*}
Then we consider the following truncated functional defined in $H_0^1(\Omega)$
$$I_{\psi}(u)=\frac{1}{2}\|u\|^2+\frac{b}{4}\|u\|^4-\frac{\lambda}{2}\|u\|_2^2+\frac{\mu}{2}\|u\|_2^2
-\frac{\mu}{2}\int_{\Omega}u^2\log{u^2}\mathrm{d}x-\frac{1}{6}\psi(\|u\|^2)\|u\|_{6}^{6}.$$
We can easily verify that $I_{\psi}(u)$ is well defined and is a $C^1$ functional in $H_0^1(\Omega)$. Its first Fr\'{e}chet
derivative is denoted by
\begin{align*}
\langle I_{\psi}'(u),\phi\rangle=&(1+b\|u\|^2)\int_{\Omega}\nabla u\nabla \phi\mathrm{d}x-\lambda\int_{\Omega}u\phi\mathrm{d}x
-\mu\int_{\Omega}u\phi\log u^2\mathrm{d}x\nonumber\\
&-\frac{1}{3}\psi'(\|u\|^2)\|u\|_{6}^{6}\int_{\Omega}\nabla u\nabla \phi\mathrm{d}x-\psi(\|u\|^2)\int_{\Omega} u^5\phi\mathrm{d}x,
\ \ \ \ \forall\,u,\phi\in H_0^1(\Omega).
\end{align*}

The following lemma shows that the truncated functional $I_{\psi}(u)$ is bounded from below in $H_0^1(\Omega)$, as desired.

\begin{lemma}\label{3-bounded-from-below}
Assume that $b>0$, $\lambda\in \mathbb{R}$ and $\mu<0$. Then $I_{\psi}(u)$ is bounded from below in $H_0^1(\Omega)$.
\end{lemma}
\begin{proof}
For any $u \in H_0^1(\Omega)\backslash\{0\}$, since $b>0$  and $\mu<0$, it follows from \eqref{S} and \eqref{log-3-1} that
\begin{align}\label{3-2-ineq1}
I_{\psi}(u)=&\frac{1}{2}\|u\|^2+\frac{b}{4}\|u\|^4-\frac{\lambda}{2}\|u\|_2^2+\frac{\mu}{2}\|u\|_2^2
-\frac{\mu}{2}\int_{\Omega}u^2\log{u^2}\mathrm{d}x-\frac{1}{6}\psi(\|u\|^2)\|u\|_{6}^{6}\nonumber\\
\geq&\frac{1}{2}\|u\|^2+\frac{\mu}{2}e^{-\frac{\lambda}{\mu}}|\Omega|-\frac{1}{6S^3}\psi(\|u\|^2)\|u\|^{6}.
\end{align}
Set
\begin{equation}\label{G-t}
\widetilde{G}(t)=\frac{1}{2}t^2+\frac{\mu}{2}e^{-\frac{\lambda}{\mu}}|\Omega|-\frac{1}{6S^{3}}\psi(t^2)t^{6}, \ \ \ t\geq0.
\end{equation}
According to the definition of $\psi$, we can see that $\widetilde{G}(t)=G(t)$ for $0\leq t \leq t_0$,
$\widetilde{G}(t)\geq G(t)$ for $t_0\leq t \leq t_{max}$ and $\widetilde{G}(t)=G_{max}$ for $t\geq t_{max}$.
Here $G(t)$ is the function defined in \eqref{G}.
Therefore, $I_{\psi}(u)\geq\widetilde{G}(\|u\|)\geq\widetilde{G}(0)=\frac{\mu}{2}e^{-\frac{\lambda}{\mu}}|\Omega|$.
This completes the proof.
\end{proof}

Next, we prove that $I_{\psi}$ satisfies the local $(PS)$ condition (known as $(PS)_c$ condition). A key point
in this process is the restriction of the critical level $c$.

\begin{lemma}\label{3-PS-condition}
Assume that $(b,\lambda,\mu)\in \mathcal{A}_2$. If
$$c<c^*=\min\left\{\frac{1}{6}S^{\frac{3}{2}}+\frac{\mu}{4}e^{-\frac{\lambda}{\mu}}|\Omega|,
\frac{1}{3}C_{b,S}+\frac{b}{12}C_{b,S}^2+\frac{\mu}{3}e^{\frac{1}{2}-\frac{\lambda}{\mu}}|\Omega|\right\},$$
where $C_{b,S}=\frac{1}{2}\left(bS^3+\sqrt{b^2S^6+4S^3}\right)$, then $I_{\psi}$ satisfies the $(PS)_c$ condition,
i.e., any sequence $\{u_n\}$ satisfying $I_{\psi}(u_n)\rightarrow c$ and $I_{\psi}'(u_n)\rightarrow 0$ as $n\rightarrow\infty$
has a strongly convergent subsequence.
\end{lemma}

\begin{proof}
Since $(b,\lambda,\mu)\in \mathcal{A}_2$, we know $c^*>0$. Assume that $\{u_n\}\subset H_0^1(\Omega)$ is a $(PS)_c$
sequence of $I_{\psi}$, i.e., $I_{\psi}(u_n)\rightarrow c$ and $I_{\psi}'(u_n)\rightarrow 0$ as $n\rightarrow\infty$.
Since $c<\frac{1}{6}S^{\frac{3}{2}}+\frac{\mu}{4}e^{-\frac{\lambda}{\mu}}|\Omega|=G_0$, we have $I_{\psi}(u_n)< G_0$
for $n$ large enough. Combining this with \eqref{3-2-ineq1} and the definition of $\psi(x)$, we can deduce that
$\{u_n\}$ is bounded in $H_0^1(\Omega)$ and $\psi(\|u_n\|^2)=1$ and $\psi'(\|u_n\|^2)=0$ for $n$ large enough.
Consequently, there exist a subsequence of $\{u_n\}$ (still
denoted by $\{u_n\}$) and a $u\in H_0^1(\Omega)$ such that, as $n\rightarrow\infty$,
\begin{equation}\label{3-2-un}
\begin{cases}
u_n&\rightharpoonup \ u \ in \ H_0^1(\Omega),\\
u_n&\rightarrow \ u \ in \ L^p(\Omega), \ \ \  1\leq p<6,\\
u_n^5&\rightharpoonup u^5 \ in \ L^{\frac{6}{5}}(\Omega),\\
u_n&\rightarrow \ u \ a.e. \ in \ \Omega.
\end{cases}
\end{equation}
According to Lemma \ref{CCP}, there exist an at most countable index set $\mathcal{J}$ and a collection of points
$\{x_k\}_{k\in \mathcal{J}}$ such that
\allowdisplaybreaks \begin{align*}
&|\nabla u_n|^2\rightharpoonup\mu\geq |\nabla u|^2+\sum_{k\in \mathcal{J}}\mu_k \delta_{x_k}, \ \ \mu_k>0,\nonumber\\
&|u_n|^{6}\rightharpoonup \nu=|u|^{6}+\sum_{k\in \mathcal{J}}\nu_k \delta_{x_k}, \ \ \nu_k>0,
\end{align*}
in the measure sense.
In addition, we also have
\begin{equation}\label{uv}
\mu_k \geq S\nu_k^{\frac{1}{3}}, \ \ \ \ k\in \mathcal{J}.
\end{equation}

Now we claim that $\mathcal{J}=\emptyset$. Suppose on the contrary that $\mathcal{J}\neq\emptyset$ and fix $k\in \mathcal{J}$.
Define a function $\phi\in C_0^\infty(\Omega)$ with $\phi(x)=1$ in $B(x_k,\rho)$, $\phi(x)=0$ in
$\Omega\backslash B(x_k,2\rho)$, $0\leq \phi(x)\leq1$ otherwise, and with $|\nabla\phi|\leq\frac{2}{\rho}$ in $\Omega$,
where $\rho >0$ is a constant.
Since $I_{\psi}'(u_n)\rightarrow0$ in $H^{-1}(\Omega)$ as $n\rightarrow\infty$ and $\{u_n\phi\}$ is bounded in $H_0^1(\Omega)$,
recalling the fact that $\psi(\|u_n\|^2)=1$ and $\psi'(\|u_n\|^2)=0$ for $n$ large enough
and using \eqref{3-11-equation1}, we have, for $n$ large enough,
\begin{align}\label{3-2-equation}
0=&\lim\limits_{n\rightarrow\infty}\langle I_{\psi}'(u_n),u_n\phi\rangle\nonumber\\
=&\lim\limits_{n\rightarrow\infty}\Bigg\{\left(1+b\int_{\Omega}|\nabla u_n|^2\mathrm{d}x\right)\int_{\Omega}\nabla {u_n}\nabla (u_n \phi)\mathrm{d}x-\lambda\int_{\Omega}u_n^2 \phi\mathrm{d}x\nonumber\\
&-\mu\int_{\Omega}u_n^2 \phi\log u_n^2\mathrm{d}x-\frac{1}{3}\psi'(\|u_n\|^2)\|u_n\|_6^6\int_{\Omega}\nabla u_n \nabla(u_n \phi)\mathrm{d}x
-\psi(\|u_n\|^2)\int_{\Omega} u_n^6 \phi\mathrm{d}x\Bigg\}\nonumber\\
=&\lim\limits_{n\rightarrow\infty}\Bigg\{\left(1+b\int_{\Omega}|\nabla u_n|^2\mathrm{d}x\right)\int_{\Omega}|\nabla {u_n}|^2\phi\mathrm{d}x
+\left(1+b\int_{\Omega}|\nabla u_n|^2\mathrm{d}x\right)\int_{\Omega}\nabla {u_n}u_n\nabla \phi\mathrm{d}x\nonumber\\
&-\lambda\int_{\Omega}u_n^2 \phi\mathrm{d}x-\mu\int_{\Omega}u_n^2 \phi\log u_n^2\mathrm{d}x
-\int_{\Omega} u_n^6 \phi\mathrm{d}x\Bigg\}\nonumber\\
\geq&\lim\limits_{n\rightarrow\infty}\Bigg\{\left(1+b\int_{\Omega}|\nabla u_n|^2\phi\mathrm{d}x\right)
\int_{\Omega}|\nabla {u_n}|^2\phi\mathrm{d}x
-\int_{\Omega} u_n^6 \phi\mathrm{d}x\Bigg\}+o(1),
\end{align}
where $o(1)\rightarrow0$ as $\rho\rightarrow0$.
Letting $\rho\rightarrow 0$ in \eqref{3-2-equation}, we have
$$0\geq(1+b\mu_k)\mu_k-\nu_k.$$
Combing this with \eqref{uv}, we have
\begin{align}\label{3-2-Ck}
\mu_k\geq C_{b,S}.
\end{align}
Since $I_{\psi}(u_n)\rightarrow c$ and $I'_{\psi}(u_n)\rightarrow0$ in $H^{-1}(\Omega)$ as $n\rightarrow\infty$,  by using the boundedness of $\{u_n\}$,
\eqref{ineq-3-1} and the fact that $\psi(\|u_n\|^2)=1$ and $\psi'(\|u_n\|^2)=0$ for $n$ large enough, one has
\begin{align}\label{3-2-ineq2}
c=&\lim\limits_{n\rightarrow\infty}\left\{I_{\psi}(u_n)-\frac{1}{6}\langle I_{\psi}'(u_n),u_n\rangle\right\}\nonumber\\
=&\frac{1}{2}\|u_n\|^2+\frac{b}{4}\|u_n\|^4-\frac{\lambda}{2}\|u_n\|_2^2+\frac{\mu}{2}\|u_n\|_2^2
-\frac{\mu}{2}\int_{\Omega}u_n^2\log{u_n^2}\mathrm{d}x\nonumber\\
&-\frac{1}{6}\psi(\|u_n\|^2)\|u_n\|_{6}^{6}-\frac{1}{6}\|u_n\|^2-\frac{b}{6}\|u_n\|^4
+\frac{\lambda}{6}\|u_n\|_2^2+\frac{\mu}{6}\int_{\Omega}u_n^2\log{u_n^2}\mathrm{d}x\nonumber\\
&+\frac{1}{18}\psi'(\|u_n\|^2)\|u_n\|^2\|u_n\|_{6}^{6}+\frac{1}{6}\psi(\|u_n\|^2)\|u_n\|_{6}^{6}\nonumber\\
=&\lim\limits_{n\rightarrow\infty}\bigg\{\frac{1}{3}\|u_n\|^2+\frac{b}{12}\|u_n\|^4-\frac{\lambda}{3}\|u_n\|_2^2+\frac{\mu}{2}\|u_n\|_2^2
-\frac{\mu}{3}\int_{\Omega}u_n^2\log{u_n^2}\mathrm{d}x\bigg\}\nonumber\\
\geq&\lim\limits_{n\rightarrow\infty}\left\{\frac{1}{3}\int_{\Omega}|\nabla u_n|^2\phi\mathrm{d}x
+\frac{b}{12}\left(\int_{\Omega}|\nabla u_n|^2\phi\mathrm{d}x\right)^2+\frac{\mu}{3}e^{\frac{1}{2}-\frac{\lambda}{\mu}}|\Omega|\right\}.
\end{align}
Letting $\rho\rightarrow0$ in \eqref{3-2-ineq2} and recalling \eqref{3-2-Ck}, one has
\begin{align*}
c\geq&\frac{1}{3}\mu_k+\frac{b}{12}\mu_k^2+\frac{\mu}{3}e^{\frac{1}{2}-\frac{\lambda}{\mu}}|\Omega|\\
\geq&\frac{1}{3}C_{b,S}+\frac{b}{12}C_{b,S}^2+\frac{\mu}{3}e^{\frac{1}{2}-\frac{\lambda}{\mu}}|\Omega|,
\end{align*}
which contradicts with the assumption that $c<\frac{1}{3}C_{b,S}+\frac{b}{12}C_{b,S}^2+\frac{\mu}{3}e^{\frac{1}{2}-\frac{\lambda}{\mu}}|\Omega|$. Consequently, $\mathcal{J}=\emptyset$ and
\begin{equation}\label{3-2-eq1}
\int_{\Omega}|u_n|^{6}\mathrm{d}x\rightarrow\int_{\Omega}|u|^{6}\mathrm{d}x, \ \ as \ \ n\rightarrow \infty.
\end{equation}

Finally, by combining $I_{\psi}'(u_n)\rightarrow 0$ in $H^{-1}(\Omega)$, the boundedness of $\{u_n\}$, and the fact
that $\psi(\|u_n\|^2)=1$ and $\psi'(\|u_n\|^2)=0$ for $n$ large enough with \eqref{3-2-un} and \eqref{3-2-eq1} and
applying Lemma \ref{2-convergence}, we have, for $n$ large enough,
\begin{align}\label{un-u=0}
o_n(1)=&\langle I_{\psi}'(u_n),u_n-u\rangle\nonumber\\
=&(1+b\|u_n\|^2)\int_{\Omega}\nabla u_n\nabla(u_n-u)\mathrm{d}x-\lambda\int_{\Omega}u_n(u_n-u)\mathrm{d}x
-\mu\int_{\Omega}u_n(u_n-u)\log u_n^2\mathrm{d}x\nonumber\\
&-\frac{1}{3}\psi'(\|u_n\|^2)\|u_n\|_{6}^{6}\int_{\Omega}\nabla u_n\nabla(u_n-u)\mathrm{d}x-\psi(\|u_n\|^2)\int_{\Omega} u_n^5(u_n-u)\mathrm{d}x\nonumber\\
=&(1+b\|u_n\|^2)\int_{\Omega}\nabla u_n\nabla(u_n-u)\mathrm{d}x-\lambda\int_{\Omega}u_n(u_n-u)\mathrm{d}x
-\mu\int_{\Omega}u_n(u_n-u)\log u_n^2\mathrm{d}x\nonumber\\
&-\int_{\Omega} u_n^5(u_n-u)\mathrm{d}x\nonumber\\
=&(1+b\|u_n\|^2)\int_{\Omega}\nabla u_n\nabla(u_n-u)\mathrm{d}x+o_n(1),
\end{align}
Then we derive that
$$\int_{\Omega}|\nabla u_n|^2\mathrm{d}x-\int_{\Omega}|\nabla u|^2\mathrm{d}x=o(1), \ \ as \ n\rightarrow\infty,$$
which, together with $u_n\rightharpoonup u \ in \ H_0^1(\Omega)$, implies that $u_n\rightarrow u \ in \ H_0^1(\Omega)$
as $n\rightarrow\infty$. This completes the proof.
\end{proof}

In order to apply the symmetric mountain pass lemma, it remans to prove that the condition $(A_2)$ in
Lemma \ref{symmetric mountain-pass lemma} holds.

\begin{lemma}\label{3-2-Ak}
Assume that $b>0$, $\lambda\in \mathbb{R}$ and $\mu<0$. Then for each $k\in \mathbb{N}$, there exists an
$A_k\in \mathcal{E}_k$ such that $\sup\limits_{u\in A_k}I_{\psi}(u)<0$.
\end{lemma}

\begin{proof}
For any $k\in \mathbb{N}$, suppose that $H_k$ is a $k$-dimensional subspace of $H_0^1(\Omega)$. According to
$\mu<0$ and $0\leq\psi(x)\leq1$ on $[0,+\infty)$, by using \eqref{2-log3} with $0<\delta<4$, the Sobolev
embedding theorem and the fact that $C_1\|u\|^2\leq\|u\|_2^2\leq C_2\|u\|^2$ in $H_k$,  one has, for any $u\in H_k\backslash \{0\}$ with $\|u\|<1$,
\begin{align*}
I_{\psi}(u)=&\frac{1}{2}\|u\|^2+\frac{b}{4}\|u\|^4-\frac{\lambda}{2}\|u\|_2^2+\frac{\mu}{2}\|u\|_2^2
-\frac{\mu}{2}\int_{\Omega}u^2\log{u^2}\mathrm{d}x-\frac{1}{6}\psi(\|u\|^2)\|u\|_{6}^{6}\\
=&\frac{1}{2}\|u\|^2+\frac{b}{4}\|u\|^4-\frac{\lambda}{2}\|u\|^2\|v\|_2^2+\frac{\mu}{2}\|u\|^2\|v\|_2^2
-\frac{\mu}{2}\|u\|^2\int_{\Omega}v^2\log{v^2}\mathrm{d}x\\
&-\frac{\mu}{2}\|u\|^2\|v\|_2^2\log{\|u\|^2}-\frac{1}{6}\psi(\|u\|^2)\|u\|^{6}\|v\|_{6}^{6}\\
\leq&\frac{1}{2}\|u\|^2+\frac{b}{4}\|u\|^4+\frac{|\lambda|}{2\lambda_1}\|u\|^2-\frac{\mu}{2}C\|u\|^2\|v\|_{2+\delta}^{2+\delta}
-\frac{\mu}{2}\|u\|^2\|v\|_2^2\log{\|u\|^2}\\
\leq&\frac{1}{2}\|u\|^2+\frac{b}{4}\|u\|^4+\frac{|\lambda|}{2\lambda_1}\|u\|^2-\frac{\mu}{2}C\|u\|^2\|v\|^{2+\delta}
-\frac{\mu}{2}C\|u\|^2\log{\|u\|^2}\\
=&\frac{\|u\|^2}{2}\left(1+\frac{b}{2}\|u\|^2+\frac{|\lambda|}{\lambda_1}-\mu C-\mu C\log{\|u\|^2}\right),
\end{align*}
where $v:=u/\|u\|$ and $\lambda_1$ is the first eigenvalue of $-\Delta$ with the Dirichlet boundary condition.
Noticing that $\mu<0$, then for any $k\in \mathbb{N}$, there exist $\rho_k$ small enough and constant $M_k>0$
such that $I_{\psi}(u)\leq-M_k<0$ when $u\in H_k$ with $\|u\|=\rho_k$.
Set $A_k=\{u\in H_k: \|u\|=\rho_k\}$. It follows from Proposition \ref{genus-proposition} $(4)$ that $\gamma(A_k)=k$. Thus, $A_k\in \mathcal{E}_k.$
This completes the proof.
\end{proof}

{\bf The proof of Theorem \ref{th1.3}.} Assume that $N=3$ and $(b, \lambda,\mu)\in \mathcal{A}_2$.
It is easy to see that $I_{\psi}\in C^1(H_0^1(\Omega),\mathbb{R})$, $I_\psi(0)=0$ and $I_\psi$ is even in $H_0^1(\Omega)$. Then according to
Lemmas \ref{3-bounded-from-below}, \ref{3-PS-condition} and \ref{3-2-Ak}, the conditions $(A_1)$ and $(A_2)$ of Lemma \ref{symmetric mountain-pass lemma} are
satisfied. Hence, it follows from Lemma \ref{symmetric mountain-pass lemma} that $I_\psi$ has a sequence of critical points $\{u_k\}$ converging to zero
with $I_{\psi}(u_k)\leq0$ and $u_k\neq0$. Since $I_{\psi}(u_k)\leq0$, $\widetilde{G}(\|u_k\|)\leq0$ and $\|u_k\|<t_0$, where $\widetilde{G}(t)$ is given in \eqref{G-t}. Recalling the definition of $\psi$, we can deduce that $\{u_k\}$ are also the critical points of $I$. The proof of Theorem \ref{th1.3} is complete.

\par
\section{The case $N=4$}
\setcounter{equation}{0}

In this section, we consider the existence of solutions to problem \eqref{eq1} when $N=4$.
From \eqref{energy} one sees that the critical exponent is exactly the same as the exponent of
the Kirchhoff term in this case, and the interaction between $\|u\|^4$ and $\|u\|_4^4$
results in some interesting phenomena. When $bS^2-1\geq0$, the nonlocal term dominates the
critical one and $I(u)$ is coercive in the whole space $H_0^1(\Omega)$. Conversely,
when $bS^2-1<0$, the critical term plays a dominant role and $I(u)$ is not bounded from below.
Therefore, in this section, we will investigate the existence of infinitely many solutions
to problem \eqref{eq1} under two cases: $bS^2-1\geq0$ and $bS^2-1<0$.

\subsection{The case $bS^2-1\geq0$}

In this subsection, we deal with the case $bS^2-1\geq0$. To the first, we shows that the functional
$I(u)$ is coercive in $H_0^1(\Omega)$.

\begin{lemma}\label{le4.1}
Assume that $bS^2-1\geq0$, $\lambda\in \mathbb{R}$ and $\mu<0$. Then $I(u)$ is coercive in $H_0^1(\Omega)$, that is
\begin{align*}
\lim\limits_{\|u\|\rightarrow+\infty}I(u)=+\infty.
\end{align*}
\end{lemma}

\begin{proof}
Since $bS^2-1\geq0$, $\lambda\in \mathbb{R}$ and $\mu<0$, using \eqref{S} and \eqref{log-3-1}, one has
\begin{align}\label{4-ineq1}
I(u)=&\frac{1}{2}\|u\|^2+\frac{b}{4}\|u\|^4-\frac{\lambda}{2}\|u\|_2^2+\frac{\mu}{2}\|u\|_2^2
-\frac{\mu}{2}\int_{\Omega}u^2\log{u^2}\mathrm{d}x
-\frac{1}{4}\|u\|_4^4\nonumber\\
\geq&\frac{1}{2}\|u\|^2+\frac{bS^2-1}{4S^2}\|u\|^4+\frac{\mu}{2}e^{-\frac{\lambda}{\mu}}|\Omega|,
\end{align}
which implies that $I(u)$ is coercive in $H_0^1(\Omega)$. This completes the proof.
\end{proof}

Next, we introduce a global compactness result for $I(u)$ which is also of independent interest.

\begin{lemma}\label{le4.2}
Assume that $bS^2-1\geq0$, $\lambda\in \mathbb{R}$ and $\mu<0$. Then $I(u)$ satisfies the global
$(PS)$ condition, i.e., any sequence $\{u_n\}$ satisfying $I(u_n)$ is bounded and $I'(u_n)\rightarrow 0$
as $n\rightarrow\infty$ has a strongly convergent subsequence.
\end{lemma}

\begin{proof}
Assume that $\{u_n\}$ is a sequence such that $\{I(u_n)\}$ is bounded and $I'(u_n)\rightarrow 0$
in $H^{-1}(\Omega)$ as $n\rightarrow\infty$. Since $I(u)$ is coercive in $H_0^1(\Omega)$,
$\{u_n\}$ is bounded in $H_0^1(\Omega)$. Consequently, passing to a subsequence if necessary,
there exist a $u\in H_0^1(\Omega)$ such that, as $n\rightarrow\infty$,
\begin{equation}\label{4-1-un-convergence}
\begin{cases}
u_{n}\rightharpoonup u \ in \ H_{0}^1(\Omega),\\
u_{n}\rightarrow u \ in \ L^p(\Omega), \ \ 1\leq p<4,\\
u_n^3\rightharpoonup u^3 \ in \ L^{\frac{4}{3}}(\Omega),\\
u_{n}\rightarrow u \ \ a.e.\ in \ \Omega.
\end{cases}
\end{equation}
To prove $u_n\rightarrow u \ in \ H_0^1(\Omega)$ as $n\rightarrow\infty$, we set $w_n=u_n-u$.
Obviously, $\{w_n\}$ is also bounded in $H_0^1(\Omega)$, and there exists a subsequence of
$\{w_{n}\}$ (still denoted by $\{w_{n}\}$) such that
\begin{align}\label{4-1-wn1}
\lim\limits_{n\rightarrow\infty}\|w_n\|^2=l\geq0.
\end{align}
From $u_{n}\rightharpoonup u$ in $H_{0}^1(\Omega)$ as $n\rightarrow\infty$, there holds
\begin{align}\label{4-1-wn2}
\|u_n\|^{2}=\| w_n\|^2+\| u\|^{2}+o_n(1), \qquad  n\rightarrow\infty.
\end{align}
Combining this with \eqref{2-log-convergence2} and \eqref{4-1-un-convergence}, we have
\begin{align}\label{4-1-Iun-u}
o_n(1)=&\langle I'(u_n),u\rangle\nonumber\\
=&\left(1+b\|u_n\|^2\right)\int_{\Omega}\nabla u_n\nabla u\mathrm{d}x-\lambda\int_{\Omega}u_n u\mathrm{d}x-\mu\int_{\Omega}u_n u\log u_n^2\mathrm{d}x-
\int_{\Omega}u_n^3 u\mathrm{d}x\nonumber\\
=&\|u\|^2+b\|u_n\|^2\|u\|^2-\lambda\|u\|_2^2-\mu\int_{\Omega}u^2\log u^2\mathrm{d}x-\|u\|_4^4+o_n(1)\nonumber\\
=&\|u\|^2+b\|u\|^4+b\|w_n\|^2\|u\|^2-\lambda\|u\|_2^2-\mu\int_{\Omega}u^2\log u^2\mathrm{d}x-\|u\|_4^4+o_n(1).
\end{align}
Then according to $I'(u_n)\rightarrow 0$  as $n\rightarrow\infty$ and the boundedness of $\{u_n\}$, using Lemma \ref{Brezis-Lieb}, \eqref{S},
\eqref{2-log-convergence1}, \eqref{4-1-un-convergence}, \eqref{4-1-wn2} and \eqref{4-1-Iun-u}, one has
\begin{align}\label{4-1-Iun-un}
o_n(1)=&\langle I'(u_n),u_n\rangle\nonumber\\
=&\left(1+b\|u_n\|^2\right)\|u_n\|^2-\lambda \|u_n\|_2^2-\mu\int_{\Omega}u_n^2\log u_n^2\mathrm{d}x-\|u_n\|_4^4\nonumber\\
=&\|w_n\|^2+\|u\|^2+b\|w_n\|^4+2b\|w_n\|^2\|u\|^2+b\|u\|^4-\lambda\|u\|_2^2-\mu\int_{\Omega}u^2\log u^2\mathrm{d}x\nonumber\\
&-\|w_n\|_4^4-\|u\|_4^4+o_n(1)\nonumber\\
=&\|w_n\|^2+b\|w_n\|^4+b\|w_n\|^2\|u\|^2-\|w_n\|_4^4+o_n(1)\nonumber\\
\geq&\|w_n\|^2+b\|w_n\|^4+b\|w_n\|^2\|u\|^2-\frac{1}{S^2}\|w_n\|^4+o_n(1).
\end{align}
Recalling \eqref{4-1-wn1} and letting $n\rightarrow\infty$ in \eqref{4-1-Iun-un}, we deduce that
\begin{align*}
0\geq l+\frac{bS^2-1}{S^2}l^2+bl\|u\|^2,
\end{align*}
which ensures $l=0$ since $bS^2-1\geq0$. Therefore, $u_n\rightarrow u$ in $H_0^1(\Omega)$ as
$n\rightarrow\infty$. This completes the proof.
\end{proof}

\begin{lemma}\label{4-1-Ak}
Assume that $bS^2-1\geq0$, $\lambda\in \mathbb{R}$ and $\mu<0$. Then for each $k\in \mathbb{N}$,
there exists an $A_k\in \mathcal{E}_k$ such that $\sup\limits_{u\in A_k}I(u)<0$.
\end{lemma}
\begin{proof}
The proof is similar to that of Lemma \ref{3-2-Ak}. Here we omit it.
\end{proof}

{\bf The proof of Theorem \ref{th1.4} with condition $(i)$.}
Assume that $N=4$, $bS^2-1\geq0$, $\lambda\in \mathbb{R}$ and $\mu<0$. It is obvious that
$I\in C^1(H_0^1(\Omega),\mathbb{R})$, $I(0)=0$ and $I$ is even in $H_0^1(\Omega)$, which,
together with Lemmas \ref{le4.1}, \ref{le4.2} and \ref{4-1-Ak}, imply that the conditions
$(A_1)$ and $(A_2)$ of Lemma \ref{symmetric mountain-pass lemma-1} are
vaild. Therefore, by Lemma \ref{symmetric mountain-pass lemma-1}, $I$ has a sequence of
solutions $\{u_k\}$ converging to zero with $I(u_k)\leq0$ and $u_k\neq0$. This completes
the proof of Theorem \ref{th1.4} with condition $(i)$.

\subsection{The case $bS^2-1<0$}

In this subsection, we consider the case $bS^2-1<0$. Different from the case $bS^2-1\geq0$,
the energy functional is on longer bounded from below and fails to satisfy the global $(PS)$
condition. So the method used in subsection 4.1 can not be applied to this case directly.
In order to overcome this difficulty, we need to introduce proper truncation on the critical term.

Since $\lambda\in \mathbb{R}$, $\mu<0$, $b>0$ and $bS^2-1<0$, similarly to the proof of \eqref{4-ineq1},
we have
\begin{align*}
I(u)&=\frac{1}{2}\|u\|^2+\frac{b}{4}\|u\|^4-\frac{\lambda}{2}\|u\|_2^2+\frac{\mu}{2}\|u\|_2^2
-\frac{\mu}{2}\int_{\Omega}u^2\log{u^2}\mathrm{d}x
-\frac{1}{4}\|u\|_{4}^{4}\nonumber\\
&\geq \frac{1}{2}\|u\|^2-\frac{1-bS^2}{4S^2}\|u\|^4+\frac{\mu}{2}e^{-\frac{\lambda}{\mu}}|\Omega|.
\end{align*}
Set
\begin{align*}
M(t)=\frac{1}{2}t^2-\frac{1-bS^2}{4S^2}t^4+\frac{\mu}{2}e^{-\frac{\lambda}{\mu}}|\Omega|, \ \ \ t\geq0,
\end{align*}
Direct calculation shows that $M'(t)>0$ for $0<t<\widetilde{t}$, $M'(t)<0$ for $t>\widetilde{t}$ and
$M(t)$ takes its maximum at $\widetilde{t}$, where $\widetilde{t}:=\sqrt{\frac{S^2}{1-bS^2}}$.
Set
\begin{eqnarray*}
\widetilde{M}:=M(\widetilde{t})=\frac{S^2}{4(1-bS^2)}+\frac{\mu}{2}e^{-\frac{\lambda}{\mu}}|\Omega|.
\end{eqnarray*}
Since $(b,\lambda,\mu)\in \mathcal{B}_0$, it follows from $\dfrac{S^2}{4(1-bS^2)}+\dfrac{\mu}{4}e^{1-\frac{\lambda}{\mu}}|\Omega|>0$
that $\dfrac{S^2}{4(1-bS^2)}+\dfrac{\mu}{2}e^{-\frac{\lambda}{\mu}}|\Omega|>0$.
Then $M_0=\dfrac{\widetilde{M}}{2}=\dfrac{S^2}{8(1-bS^2)}+\dfrac{\mu}{4}e^{-\frac{\lambda}{\mu}}|\Omega|>0$
and there exists a unique $t_0\in(0,\widetilde{t})$ such that $M(t_0)=M_0$.

Define a truncation function $\zeta(x)\in C^1[0,\infty)$ such that
\begin{equation*}
\zeta(x)=\begin{cases}
1, \ \ \ \ \  &0\leq x \leq t_0^2,\\
0\leq\zeta(x)\leq1,&t_0^2\leq x \leq \widetilde{t}^2,\\
\frac{2S^2}{x}+\frac{2\mu S^2 e^{-\frac{\lambda}{\mu}}|\Omega|-4S^2\widetilde{M}}{x^2}+bS^2, &x\in [\widetilde{t}^2,+\infty),
\end{cases}
\end{equation*}
and consider the following truncated functional defined in $H_0^1(\Omega)$
$$I_{\zeta}(u)=\frac{1}{2}\|u\|^2+\frac{b}{4}\|u\|^4-\frac{\lambda}{2}\|u\|_2^2+\frac{\mu}{2}\|u\|_2^2
-\frac{\mu}{2}\int_{\Omega}u^2\log{u^2}\mathrm{d}x-\frac{1}{4}\zeta(\|u\|^2)\|u\|_{4}^{4}.$$
The first Fr\'{e}chet derivative of $I_{\zeta}(u)$ is denoted by
\begin{align*}
\langle I_{\zeta}'(u),\phi\rangle=&\left(1+b\|u\|^2\right)\int_{\Omega}\nabla u\nabla \phi\mathrm{d}x-\lambda\int_{\Omega}u\phi\mathrm{d}x-\mu\int_{\Omega}u\phi\log u^2\mathrm{d}x\nonumber\\
&-\frac{1}{2}\|u\|_4^4\zeta'(\|u\|^2)\int_{\Omega}\nabla u\nabla \phi\mathrm{d}x-\zeta(\|u\|^2)\int_{\Omega} u^3\phi\mathrm{d}x,
\ \ \ \ \ \forall\,u,\phi\in H_0^1(\Omega).
\end{align*}

Now, we show that the truncated energy functional $I_{\zeta}(u)$ is bounded from below in $H_0^1(\Omega)$.

\begin{lemma}\label{4-bounded-from-below}
Assume that $b>0$, $\lambda\in \mathbb{R}$, $\mu<0$ and $bS^2-1<0$. Then $I_{\zeta}(u)$ is
bounded from below in $H_0^1(\Omega)$.
\end{lemma}

\begin{proof}
For any $u \in H_0^1(\Omega)\backslash\{0\}$, since $b>0$, $\lambda\in \mathbb{R}$, $\mu<0$
and $bS^2-1<0$, according to \eqref{S}, \eqref{log-3-1} and the definition of the truncation
function $\zeta(x)$, one has
\begin{align}\label{4-Iu}
I_{\zeta}(u)=&\frac{1}{2}\|u\|^2+\frac{b}{4}\|u\|^4-\frac{\lambda}{2}\|u\|_2^2
+\frac{\mu}{2}\|u\|_2^2-\frac{\mu}{2}\int_{\Omega}u^2\log{u^2}\mathrm{d}x
-\frac{1}{4}\zeta(\|u\|^2)\|u\|_{4}^{4}\nonumber\\
\geq&\frac{1}{2}\|u\|^2+\frac{b}{4}\|u\|^4-\frac{1}{4S^2}\zeta(\|u\|^2)\|u\|^{4}
+\frac{\mu}{2}e^{-\frac{\lambda}{\mu}}|\Omega|.
\end{align}
Define
\begin{align*}
M_{\zeta}(t)=\frac{1}{2}t^2+\frac{b}{4}t^4-\frac{1}{4S^2}\zeta(t^2)t^{4}
+\frac{\mu}{2}e^{-\frac{\lambda}{\mu}}|\Omega|, \ \ \ t\geq0.
\end{align*}
From the definition of $\zeta(x)$, we can deduce that $M_{\zeta}(t)=M(t)$ for $0\leq t \leq t_0$,
$M_{\zeta}(t)\geq M(t)$ for $t_0\leq t \leq \widetilde{t}$ and
$M_{\zeta}(t)=\widetilde{M}$ for $t\geq \widetilde{t}$.
Thus, $I_{\zeta}(u)\geq M_{\zeta}(\|u\|)\geq \frac{\mu}{2}e^{-\frac{\lambda}{\mu}}|\Omega|$
for all $u\in H_0^1(\Omega)$. This completes the proof.
\end{proof}

Next we shall prove that $I_{\zeta}$ satisfies the $(PS)_c$ condition.

\begin{lemma}\label{4-PS-condition}
Assume that  $bS^2-1<0$ and $(b,\lambda,\mu)\in \mathcal{B}_0$. If
$$c<c^*:=\min\left\{\frac{S^2}{8(1-bS^2)}+\frac{\mu}{4}e^{-\frac{\lambda}{\mu}}|\Omega|,
\frac{S^2}{4(1-bS^2)}+\frac{\mu}{4}e^{1-\frac{\lambda}{\mu}}|\Omega|\right\},$$
then $I_{\zeta}$ satisfies the $(PS)_c$ condition.
\end{lemma}

\begin{proof}
Since $(b,\lambda,\mu)\in \mathcal{B}_0$, $c^*>0$.
Assume that $\{u_n\}\subset H_0^1(\Omega)$ is a $(PS)_c$ sequence of $I_{\zeta}$.
Since $c<M_0=\frac{S^2}{8(1-bS^2)}+\frac{\mu}{4}e^{-\frac{\lambda}{\mu}}|\Omega|$,
we have $I_{\zeta}(u_n)< M_0$ for $n$ large enough.
Combining this with \eqref{4-Iu} and the definition of $\zeta(x)$, we can derive that $\{u_n\}$ is bounded in $H_0^1(\Omega)$,
$\zeta(\|u_n\|^2)=1$ and $\zeta'(\|u_n\|^2)=0$ for $n$ large enough.
Therefore there exist a subsequence of $\{u_n\}$ (still denoted by $\{u_n\}$) and a $u\in H_0^1(\Omega)$
such that, as $n\rightarrow\infty$,
\begin{equation}\label{4-un-convergence2}
\begin{cases}
u_{n}\rightharpoonup u \ in \ H_{0}^1(\Omega),\\
u_{n}\rightarrow u \ in \ L^p(\Omega), \ \ 1\leq p<4,\\
u_n^3\rightharpoonup u^3 \ in \ L^{\frac{4}{3}}(\Omega),\\
u_{n}\rightarrow u \ \ a.e.\ in \ \Omega.
\end{cases}
\end{equation}
Moreover, according to Lemma \ref{CCP}, there exists an at most countable index set $\mathcal{J}$ and a collection of points $\{x_k\}_{k\in \mathcal{J}}$ such that
\allowdisplaybreaks \begin{align}\label{4-uv1}
&|\nabla u_n|^2\rightharpoonup\mu\geq |\nabla u|^2+\sum_{k\in \mathcal{J}}\mu_k \delta_{x_k}, \ \ \mu_k>0,\nonumber\\
&|u_n|^{4}\rightharpoonup \nu=|u|^{4}+\sum_{k\in \mathcal{J}}\nu_k \delta_{x_k}, \ \ \nu_k>0,
\end{align}
in the measure sense.
In addition, we also have
\begin{equation}\label{4-uv2}
\mu_k \geq S\nu_k^{\frac{1}{2}}, (k\in \mathcal{J}).
\end{equation}

Now we claim that $\mathcal{J}=\emptyset$. Suppose on the contrary that $\mathcal{J}\neq\emptyset$ and fix $k\in \mathcal{J}$.
Define a function $\phi\in C_0^\infty(\Omega)$ satisfying $\phi(x)=1$ in $B(x_k,\rho)$, $\phi(x)=0$ in
$\Omega\backslash B(x_k,2\rho)$, $0\leq \phi(x)\leq1$ otherwise and  $|\nabla\phi|\leq\frac{2}{\rho}$ in $\Omega$,
where $\rho >0$ is an appropriate constant.
Since $I_{\zeta}'(u_n)\rightarrow0$ in $H^{-1}(\Omega)$ as $n\rightarrow\infty$ and $\{u_n\phi\}$ is bounded in $H_0^1(\Omega)$, recalling the fact that $\zeta(\|u_n\|^2)=1$ and $\zeta'(\|u_n\|^2)=0$ for $n$ large enough, we have
\begin{align}\label{4-equation1}
0=&\lim\limits_{n\rightarrow\infty}\langle I_{\zeta}'(u_n),u_n\phi\rangle\nonumber\\
=&\lim\limits_{n\rightarrow\infty}\Bigg\{\left(1+b\|u_n\|^2\right)
\int_{\Omega}\nabla {u_n}\nabla (u_n \phi)\mathrm{d}x-\lambda\int_{\Omega}u_n^2 \phi\mathrm{d}x-\mu\int_{\Omega}u_n^2 \phi\log u_n^2\mathrm{d}x\nonumber\\
&-\frac{1}{2}\|u_n\|_4^4\zeta'(\|u_n\|^2)\int_{\Omega}\nabla u_n \nabla(u_n \phi)\mathrm{d}x
-\zeta(\|u_n\|^2)\int_{\Omega} u_n^4 \phi\mathrm{d}x\Bigg\}\nonumber\\
=&\lim\limits_{n\rightarrow\infty}\Bigg\{\left(1+b\|u_n\|^2\right)\int_{\Omega}\nabla {u_n}\nabla (u_n \phi)\mathrm{d}x
-\lambda\int_{\Omega}u_n^2 \phi\mathrm{d}x-\mu\int_{\Omega}u_n^2 \phi\log u_n^2\mathrm{d}x
-\int_{\Omega} u_n^4 \phi\mathrm{d}x\Bigg\}\nonumber\\
=&\lim\limits_{n\rightarrow\infty}\Bigg\{\left(1+b\|u_n\|^2\right)\int_{\Omega}|\nabla {u_n}|^2\phi\mathrm{d}x
-\int_{\Omega} u_n^4 \phi\mathrm{d}x\Bigg\}+o(1),
\end{align}
where $o(1)\rightarrow0$ as $\rho\rightarrow0$. The last equality comes from
\begin{align}\label{4-equation2}
\lim\limits_{n\rightarrow\infty}\Bigg\{\left(1+b\|u_n\|^2\right)\int_{\Omega}\nabla {u_n}u_n\nabla \phi\mathrm{d}x
-\lambda\int_{\Omega}u_n^2 \phi\mathrm{d}x-\mu\int_{\Omega}u_n^2 \phi\log u_n^2\mathrm{d}x\Bigg\}=o(1),  \ as \  \rho\rightarrow0,
\end{align}
the proof of which is similar to that of \eqref{3-11-equation1} and hence is omitted.
From \eqref{4-equation1}, we have
\begin{align*}
0\geq\lim\limits_{n\rightarrow\infty}\left\{\left(1+b\int_{\Omega}|\nabla {u_n}|^2\phi\mathrm{d}x\right)\int_{\Omega}|\nabla {u_n}|^2\phi\mathrm{d}x
-\int_{\Omega} u_n^4 \phi\mathrm{d}x\right\}+o(1).
\end{align*}
Letting $\rho\rightarrow 0$, we have
$$0\geq(1+b\mu_k)\mu_k-\nu_k.$$
Combing this with \eqref{4-uv2} and recalling $bS^2-1<0$, one has
\begin{align}\label{4-Ck}
\mu_k\geq\frac{S^2}{1-bS^2}.
\end{align}
Since $b>0$, $\lambda\in \mathbb{R}$ and $\mu<0$, by a similar estimate to \eqref{log-3-1}, we have
\begin{align}\label{4-2-log1}
-\frac{\lambda}{4}\|u_n\|_2^2+\frac{\mu}{2}\|u_n\|_2^2-\frac{\mu}{4}\int_{\Omega}u_n^2\log{u_n^2}\mathrm{d}x
\geq\frac{\mu}{4}e^{1-\frac{\lambda}{\mu}}|\Omega|.
\end{align}
Since $I_{\zeta}(u_n)\rightarrow c$ and $I'_{\zeta}(u_n)\rightarrow0$ in $H^{-1}(\Omega)$ as $n\rightarrow\infty$,  by using the boundedness of $\{u_n\}$,
\eqref{4-2-log1} and the fact that $\zeta(\|u_n\|^2)=1$ and $\zeta'(\|u_n\|^2)=0$ for $n$ large enough, we have,
\begin{align}\label{4-equation3}
c=&\lim\limits_{n\rightarrow\infty}\left\{I_{\zeta}(u_n)-\frac{1}{4}\langle I_{\zeta}'(u_n),u_n\rangle\right\}\nonumber\\
=&\lim\limits_{n\rightarrow\infty}\bigg\{\frac{1}{4}\|u_n\|^2-\frac{\lambda}{4}\|u_n\|_2^2+\frac{\mu}{2}\|u_n\|_2^2
-\frac{\mu}{4}\int_{\Omega}u_n^2\log{u_n^2}\mathrm{d}x\bigg\}\nonumber\\
\geq&\lim\limits_{n\rightarrow\infty}\left\{\frac{1}{4}\int_{\Omega}|\nabla u_n|^2\phi\mathrm{d}x
+\frac{\mu}{4}e^{1-\frac{\lambda}{\mu}}|\Omega|\right\}.
\end{align}
Letting $\rho\rightarrow0$ in \eqref{4-equation3}, it follows from \eqref{4-Ck} that
\begin{align*}
c\geq&\frac{1}{4}\mu_k+\frac{\mu}{4}e^{1-\frac{\lambda}{\mu}}|\Omega|
\geq\frac{S^2}{4(1-bS^2)}+\frac{\mu}{4}e^{1-\frac{\lambda}{\mu}}|\Omega|,
\end{align*}
which contradicts with $c<\frac{S^2}{4(1-bS^2)}+\frac{\mu}{4}e^{1-\frac{\lambda}{\mu}}|\Omega|$.
Consequently, $\mathcal{J}=\emptyset$ and
\begin{equation}\label{4-2-un4}
\int_{\Omega}|u_n|^{4}\mathrm{d}x\rightarrow\int_{\Omega}|u|^{4}\mathrm{d}x, \ \ as \ \ n\rightarrow \infty.
\end{equation}

Using $I_{\zeta}'(u_n)\rightarrow 0$ in $H^{-1}(\Omega)$ as $n\rightarrow \infty$, the boundedness of $\{u_n\}$ and the fact that $\zeta(\|u_n\|^2)=1$ and $\zeta'(\|u_n\|^2)=0$ for $n$ large enough again, together with Lemma \ref{2-convergence}, \eqref{4-un-convergence2} and \eqref{4-2-un4},  we have
\begin{align*}
o_n(1)=&\langle I_{\zeta}'(u_n),u_n-u\rangle\nonumber\\
=&\Big(1+b\|u_n\|^2\Big)\int_{\Omega}\nabla u_n\nabla(u_n-u)\mathrm{d}x-\lambda\int_{\Omega}u_n(u_n-u)\mathrm{d}x\\
&-\mu\int_{\Omega}u_n(u_n-u)\log u_n^2\mathrm{d}x-\frac{1}{2}\|u_n\|_4^4\zeta'(\|u_n\|^2)\int_{\Omega}\nabla u_n\nabla(u_n-u)\mathrm{d}x\nonumber\\
&-\zeta(\|u_n\|^2)\int_{\Omega} u_n^3(u_n-u)\mathrm{d}x\nonumber\\
=&\left(1+b\|u_n\|^2\right)\int_{\Omega}\nabla u_n\nabla(u_n-u)\mathrm{d}x-\lambda\int_{\Omega}u_n(u_n-u)\mathrm{d}x\\
&-\mu\int_{\Omega}u_n(u_n-u)\log u_n^2\mathrm{d}x-\int_{\Omega} u_n^3(u_n-u)\mathrm{d}x\nonumber\\
=&(1+b\|u_n\|^2)\int_{\Omega}\nabla u_n\nabla(u_n-u)\mathrm{d}x+o_n(1),
\end{align*}
which implies that
$$\int_{\Omega}|\nabla u_n|^2\mathrm{d}x-\int_{\Omega}|\nabla u|^2\mathrm{d}x=o(1), \ \ as \ n\rightarrow\infty.$$
Combining this with $u_n\rightharpoonup u \ in \ H_0^1(\Omega)$  as $n\rightarrow\infty$, we have $u_n\rightarrow u \ in \ H_0^1(\Omega)$  as $n\rightarrow\infty$. This completes the proof.
\end{proof}

\begin{lemma}\label{4-2-Ak}
Assume that $b>0$, $\lambda\in \mathbb{R}$ and $\mu<0$. Then for any $k\in \mathbb{N}$, there exists an $A_k\in \mathcal{E}_k$ such that $\sup\limits_{u\in A_k}I_\zeta(u)<0$.
\end{lemma}
\begin{proof}
The proof is similar to that of Lemma \ref{3-2-Ak}. Here we omit it.
\end{proof}

{\bf Proof of Theorem \ref{th1.4} with condition $(ii)$.} Assume that $N=4$, $bS^2-1<0$ and
$(b,\lambda,\mu)\in \mathcal{B}_0$. Notice that $I_{\zeta}\in C^1(H_0^1(\Omega),\mathbb{R})$,
$I_{\zeta}(0)=0$ and $I_{\zeta}(u)$ is even in $H_0^1(\Omega)$. Then by means of
Lemmas \ref{4-bounded-from-below}, \ref{4-PS-condition} and \ref{4-2-Ak}, the assumptions
$(A_1)$ and $(A_2)$ of Lemma \ref{symmetric mountain-pass lemma} are satisfied.
From Lemma \ref{symmetric mountain-pass lemma}, $I_\zeta$ has a sequence of critical
points $\{u_k\}$ converging to zero with $I_{\zeta}(u_k)\leq0$ and
$u_k\neq0$. Since $I_{\zeta}(u_k)\leq0$, then $M_{\zeta}(\|u_k\|)\leq0$ and $\|u_k\|<t_0$,
where $M_{\zeta}(t)$ is given in the proof of Lemma \ref{4-bounded-from-below}. According
to the definition of $\zeta$, we know that the sequence $\{u_k\}$ are also the critical
points of $I$. Consequently, the proof of Theorem \ref{th1.4} with condition $(ii)$ is complete.

\par
\section{The case $N\geq5$}
\setcounter{equation}{0}

In this section, we consider problem \eqref{eq1} with $N\geq5$. In this case, the critical
Sobolev exponent is strictly less than $4$, which, together with the fact that the sign of
the logarithmic term is indefinite and that the logarithmic term does not satisfy the
Ambrosetti-Rabinowitz type condition, makes it impossible to prove the compactness
of the $(PS)$ sequences. Moreover, the methods used in the above two cases
are no longer applicable. However, we notice that the energy functional possesses
the local minimum structure around $0$. In view of this, we avoid the lack of
compactness of $I(u)$ by considering the existence of a local minimum solution on an appropriate
ball in $H_0^1(\Omega)$. Furthermore, by imposing additional conditions on the parameters,
we show that the local minimum solution is also a ground state solution.

We first prove the following lemma, which shows that there exists a constant $\widetilde{r}$
such that $I(u)>0$ for all $u\in H_0^1(\Omega)$ with
$\|u\|=\widetilde{r}$.

\begin{lemma}\label{le5.1}
Assume that $(b,\lambda,\mu)\in \mathcal{C}_0$.  Then there exists $\widetilde{r}>0$ such
that $I(u)>0$ for all $u\in H_0^1(\Omega)$ with
$\|u\|=\widetilde{r}$.
\end{lemma}

\begin{proof}
For all $u\in H_0^1(\Omega)\backslash \{0\}$, since $b>0$, $\lambda\in \mathbb{R}$ and $\mu<0$,
by \eqref{S} and \eqref{log-3-1}, we have
\begin{align}\label{5-ineq1}
I(u)&=\frac{1}{2}\|u\|^2+\frac{b}{4}\|u\|^4-\frac{\lambda}{2}\|u\|_2^2+\frac{\mu}{2}\|u\|_2^2
-\frac{\mu}{2}\int_{\Omega}u^2\log{u^2}\mathrm{d}x
-\frac{1}{2^{*}}\|u\|_{2^{*}}^{2^{*}}\nonumber\\
&\geq \frac{1}{2}\|u\|^2+\frac{b}{4}\|u\|^4+\frac{\mu}{2}e^{-\frac{\lambda}{\mu}}|\Omega|
-\frac{1}{2^{*}}\frac{1}{S^{\frac{2^*}{2}}}\|u\|^{2^*}.
\end{align}
Set
\begin{align}\label{5-eq2}
h(t)=\frac{1}{2}t^2+\frac{b}{4}t^4+\frac{\mu}{2}e^{-\frac{\lambda}{\mu}}|\Omega|
-\frac{1}{2^{*}}\frac{1}{S^{\frac{2^*}{2}}}t^{2^*}, \ \ \ t\geq0,
\end{align}
and $\widetilde{r}:=\left(\frac{4}{b(N-4)}\right)^{\frac{1}{2}}$. Since $(b,\lambda,\mu)\in \mathcal{C}_0$,
we have $h(\widetilde{r})>0$.
Therefore, $I(u)\geq h(\widetilde{r})>0$ for all $u\in H_0^1(\Omega)$ with $\|u\|=\widetilde{r}$.
The proof is complete.
\end{proof}

Define
$$c_{r_0}:=\inf_{\|u\|\leq r_0}I(u),$$
where $r_0$ is chosen in such a way that $r_0<\widetilde{r}$ and $h(r_0)>0$.
We shall show that $-\infty <c_{r_0}<0$.

\begin{lemma}\label{le5.2}
Suppose that $b>0$, $\lambda\in \mathbb{R}$ and $\mu<0$. Then we have
$$-\infty <c_{r_0}<0.$$
\end{lemma}

\begin{proof}
It is obvious that $c_{r_0}>-\infty$ by \eqref{5-ineq1} and the
definition of $c_{r_0}$. To show that $c_{r_0}<0$, for any $u\in H_0^1(\Omega)\setminus\{0\}$
and $t>0$, we have
\begin{align*}
I(tu)=t^2\left(\frac{1}{2}\| u\|^{2}+\frac{b}{4}t^2\|u\|^4-\frac{\lambda}{2}\|u\|_{2}^{2}-\frac{\mu}{2}\log t^2\|u\|_{2}^{2}
+\frac{\mu}{2}\int_{\Omega}u^{2}(1-\log u^{2})\mathrm{d}x-\frac{1}{2^*}t^{2^*}\|u\|_{2^*}^{2^*}\right).
\end{align*}
Since $\mu<0$, there exists a $t_u>0$ suitably small such that $\|t_u u\|\leq r_0$ and $I(t_uu)<0$,
which implies that $c_{r_0}<0$. The proof is complete.
\end{proof}

On the basis of Lemmas \ref{le5.1} and \ref{le5.2}, we will prove the compactness of the minimizing
sequence for $c_{r_0}$, which plays a key role in the proof of Theorem \ref{th1.5}.

\begin{lemma}\label{le5.3}
Assume that $(b,\lambda,\mu)\in \mathcal{C}_0$. Then there exists a minimizing sequence
$\{u_n\}$ for $c_{r_0}$ and a $u\in H_0^1(\Omega)$ such that, up to a subsequence,
$u_n \rightarrow u$ in $H_0^1(\Omega)$ as $n\rightarrow\infty$.
\end{lemma}

\begin{proof}
Let $\{v_n\}\subset B_{r_0}$ be a minimizing sequence of $c_{r_0}$ such that $c_{r_0}\leq I(v_n)\leq c_{r_0}+\dfrac{1}{n}$ for all $n\in\mathbb{N}$.
Since $c_{r_0}<0$, it follows from \eqref{5-ineq1} that there is a constant $\sigma>0$ suitably
small such that $\|v_n\|\leq r_0-\sigma$ for all $n$ large enough. Similar to the proof of Lemma \ref{le3.3},
we can prove that there exists another minimizing sequence $\{u_n\}\subset B_{r_0}$ for $c_{r_0}$ such that
$\lim\limits_{n\rightarrow\infty}I(u_n)=c_{r_0}$ and $\lim\limits_{n\rightarrow\infty}I'(u_n)=0$.

Since $\{u_n\}$ is bounded in $H_0^1(\Omega)$, up to a subsequence, we may assume that there exists $u\in H_0^1(\Omega)$ such that,
as $n\rightarrow\infty$,
\begin{equation}\label{convergence-5}
\begin{cases}
u_{n}\rightharpoonup u \ in \ H_{0}^1(\Omega),\\
u_{n}\rightarrow u \ in \ L^p(\Omega), \ \ 1\leq p<2^*,\\
|u_n|^{2^*-2}u_n\rightharpoonup |u|^{2^*-2}u \ in \ L^{\frac{2^*}{2^*-1}}(\Omega),\\
u_{n}\rightarrow u \ \ a.e.\ in \ \Omega.
\end{cases}
\end{equation}
We claim that $u_n\rightarrow u$ in $H_0^1(\Omega)$ as $n\rightarrow \infty$. For this purpose, we set $w_n=u_n-u$.  Obviously, $\{w_n\}$ is also a bounded sequence in $H_0^1(\Omega)$. So there exists a subsequence of $\{w_{n}\}$ (still denoted by $\{w_{n}\}$) such that
\begin{align}\label{wn-5-1}
\lim\limits_{n\rightarrow\infty}\|w_n\|^2=l\geq0.
\end{align}
Using the fact $u_n\rightharpoonup u$ in $H_0^1(\Omega)$ as $n\rightarrow\infty$, we have
\begin{align}\label{wn-5-3}
\|u_n\|^{2}=\| w_n\|^2+\| u\|^{2}+o_n(1), \qquad  n\rightarrow\infty.
\end{align}
Since $I'(u_n)\rightarrow0$ as $n\rightarrow\infty$, we can deduce from \eqref{2-log-convergence2}, \eqref{convergence-5} and \eqref{wn-5-3} that, as $n\rightarrow \infty$,
\begin{align}\label{Iun-u}
o_n(1)=&\langle I'(u_n),u\rangle\nonumber\\
=&\left(1+b\|u_n\|^2\right)\int_{\Omega}\nabla u_n\nabla u\mathrm{d}x-\lambda\int_{\Omega}u_n u\mathrm{d}x-\mu\int_{\Omega}u_n u\log u_n^2\mathrm{d}x-
\int_{\Omega}|u_n|^{2^*-2}u_n u\mathrm{d}x\nonumber\\
=&\|u\|^2+b\|u\|^4+b\|w_n\|^2\|u\|^2-\lambda\|u\|_2^2
-\mu\int_{\Omega}u^2\log u^2\mathrm{d}x-\|u\|_{2^*}^{2^*}+o_n(1).
\end{align}
Combining $I'(u_n)\rightarrow0$ with the boundedness of $\{u_n\}$ and recalling \eqref{2-log-convergence1}, \eqref{convergence-5}, \eqref{wn-5-3}, \eqref{Iun-u} and Lemma \ref{Brezis-Lieb},
we have, as $n\rightarrow \infty$,
\begin{align}\label{Iun-un}
o_n(1)=&\langle I'(u_n),u_n\rangle\nonumber\\
=&\left(1+b\|u_n\|^2\right)\|u_n\|^2-\lambda \|u_n\|_2^2
-\mu\int_{\Omega}u_n^2\log u_n^2\mathrm{d}x-\|u_n\|_{2^*}^{2^*}\nonumber\\
=&\|w_n\|^2+\|u\|^2+b\|w_n\|^4+2b\|w_n\|^2\|u\|^2+b\|u\|^4-\lambda\|u\|_2^2
-\mu\int_{\Omega}u^2\log u^2\mathrm{d}x\nonumber\\
&-\|w_n\|_{2^*}^{2^*}-\|u\|_{2^*}^{2^*}+o_n(1)\nonumber\\
=&\|w_n\|^2+b\|w_n\|^4+b\|w_n\|^2\|u\|^2-\|w_n\|_{2^*}^{2^*}+o_n(1).
\end{align}
Suppose that $l>0$. Since $I(u_n)\rightarrow c_{r_0}$ as $n\rightarrow\infty$, it follows from \eqref{2-log-convergence1},
\eqref{convergence-5}, \eqref{wn-5-3},  \eqref{Iun-un} and Lemma \ref{Brezis-Lieb} that, as $n\rightarrow \infty$,
\begin{align}\label{Iun-Iu}
c_{r_0}+ o_n(1)=I(u_n)=&\frac{1}{2}\|u_n\|^2+\frac{b}{4}\|u_n\|^4-\frac{\lambda}{2}\|u_n\|_2^2+\frac{\mu}{2}\|u_n\|_2^2
-\frac{\mu}{2}\int_{\Omega}u_n^2\log{u_n^2}\mathrm{d}x-\frac{1}{2^*}\|u_n\|_{2^*}^{2^*}\nonumber\\
=&\frac{1}{2}\|w_n\|^2+\frac{1}{2}\|u\|^2+\frac{b}{4}\|w_n\|^4+\frac{b}{4}\|u\|^4+\frac{b}{2}\|w_n\|^2\|u\|^2
-\frac{\lambda}{2}\|u\|_2^2\nonumber\\
&+\frac{\mu}{2}\|u\|_2^2-\frac{\mu}{2}\int_{\Omega}u^2\log{u^2}\mathrm{d}x-\frac{1}{2^*}\|w_n\|_{2^*}^{2^*}
-\frac{1}{2^*}\|u\|_{2^*}^{2^*}+o_n(1)\nonumber\\
=&I(u)+\frac{1}{2}\|w_n\|^2+\frac{b}{4}\|w_n\|^4+\frac{b}{2}\|w_n\|^2\|u\|^2-\frac{1}{2^*}\|w_n\|_{2^*}^{2^*}
+o_n(1)\nonumber\\
=&I(u)+\frac{1}{2^*}\langle I'(u_n),u_n\rangle+\left(\frac{1}{2}-\frac{1}{2^*}\right)\|w_n\|^2
-\left(\frac{1}{2^*}-\frac{1}{4}\right)b\|w_n\|^4\nonumber\\
&+\left(\frac{1}{2}-\frac{1}{2^*}\right)b\|w_n\|^2\|u\|^2+o_n(1)\nonumber\\
=&I(u)+\left(\frac{1}{2}-\frac{1}{2^*}\right)\|w_n\|^2
-\left(\frac{1}{2^*}-\frac{1}{4}\right)b\|w_n\|^4\nonumber\\
&+\left(\frac{1}{2}-\frac{1}{2^*}\right)b\|w_n\|^2\|u\|^2+o_n(1).
\end{align}
Letting $n\rightarrow\infty$ in \eqref{Iun-Iu} and applying \eqref{wn-5-1}, one has
\begin{align}\label{equation1}
c_{r_0}=I(u)+\left(\frac{1}{2}-\frac{1}{2^*}\right)l
-\left(\frac{1}{2^*}-\frac{1}{4}\right)bl^2
+\left(\frac{1}{2}-\frac{1}{2^*}\right)bl\|u\|^2.
\end{align}
By the weak lower semi-continuity of the norm, we have $\|u\|\leq r_0$,
which, together with \eqref{wn-5-1} and \eqref{wn-5-3}, guarantees
\begin{equation}\label{l-upper}
0< l\leq r_0^2<\widetilde{r}^2=\frac{4}{b(N-4)}.
\end{equation}
Since $2<2^*<4$, by a direct calculation,
we can deduce from \eqref{l-upper} that
\begin{align}\label{5-equation1}
c_{r_0}=I(u)+\left(\frac{1}{2}-\frac{1}{2^*}\right)l
-\left(\frac{1}{2^*}-\frac{1}{4}\right)bl^2
+\left(\frac{1}{2}-\frac{1}{2^*}\right)bl\|u\|^2>I(u),
\end{align}
a contradiction. Therefore, $l=0$, i.e., $u_n\rightarrow u$ in $H_0^1(\Omega)$ as
$n\rightarrow\infty$. This completes the proof.
\end{proof}

{\bf The proof of Theorem \ref{th1.5}.}
Suppose that $N\geq 5$ and $(b,\lambda,\mu)\in \mathcal{C}_0$.
From Lemma \ref{le5.3} and the fact that $I(u)$ is a $C^1$ functional in $H_0^1(\Omega)$,
one obtains that $u$ is a local minimum solution to problem \eqref{eq1} with $I(u)<0$.

In addition to the assumptions of Theorem \ref{th1.5}, suppose that
$b\geq b^*=\frac{2}{N-2}(\frac{N-4}{N-2})^{\frac{N-4}{2}}S^{-\frac{N}{2}}$.
Direct calculation shows that $h'(t)\geq0$ on $[0,+\infty)$, which, together with $h(r_0)>0$,
implies that $I(u)\geq h(\|u\|)>0$ for $\|u\|\geq r_0$, where $h(t)$ is given by \eqref{5-eq2}.
Consequently, the local minimum solution
$u$ with negative energy must be a ground state solution.
This completes the proof of Theorem \ref{th1.5}.

\vskip5mm

{\bf Declaration}\\
No potential competing interest was reported by the authors.

{\bf Data availability}\\
No data was used for the research described in the article.

{\bf Acknowledgements}\\
The authors would like to express their sincere gratitude to Professor Wenjie Gao in Jilin University for his enthusiastic guidance and constant encouragement.

\end{document}